\documentclass[sn-mathphys,Numbered]{sn-jnl}

\usepackage[sort&compress]{natbib}
\usepackage{graphicx}%
\usepackage{multirow}%
\usepackage{amsmath,amssymb,amsfonts,listings,mathtools,}%
\usepackage{amsthm}%
\usepackage{mathrsfs}%
\usepackage[title]{appendix}%
\usepackage{xcolor}%
\usepackage{textcomp}%
\usepackage{manyfoot}%
\usepackage{booktabs}%
\usepackage{algorithm}%
\floatname{algorithm}{Algorithm}
\usepackage{algorithmic,float}
\usepackage{listings,geometry}%
\usepackage{caption}

\theoremstyle{thmstyleone}%
\newtheorem{theorem}{Theorem}
\newtheorem{Definition}[theorem]{Definition}
\newtheorem{proposition}[theorem]{Proposition}%
\theoremstyle{thmstyletwo}%
\newtheorem{lemma}{Lemma}%
\theoremstyle{thmstylethree}%
\usepackage{footmisc}   

\begin{document}


\title[Article Title]{Low-rank generalized alternating direction implicit iteration method for solving matrix equations}





\author*[1]{\fnm{Juan} \sur{Zhang}}
\email{zhangjuan@xtu.edu.cn}
\author[2]{\fnm{Wenlu} \sur{Xun}}

\affil[1]{Key Laboratory of Intelligent Computing and Information Processing
of Ministry of Education, Hunan Key Laboratory for Computation and Simulation in Science and
Engineering, School of Mathematics and Computational Science, Xiangtan University, Xiangtan,
Hunan, China}

\affil[2]{School of Mathematics and Computational Science, Xiangtan University, Xiangtan, Hunan,
China}

\abstract{This paper presents an effective low-rank generalized alternating direction implicit iteration (R-GADI) method for solving large-scale sparse and stable Lyapunov matrix equations and continuous-time algebraic Riccati matrix equations. The method is based on generalized alternating direction implicit iteration (GADI), which exploits the low-rank property of matrices and utilizes the Cholesky factorization approach for solving. The advantage of the new algorithm lies in its direct and efficient low-rank formulation, which is a variant of the Cholesky decomposition in the Lyapunov GADI method, saving storage space and making it computationally effective. When solving the continuous-time algebraic Riccati matrix equation, the Riccati equation is first simplified to a Lyapunov equation using the Newton method, and then the R-GADI method is employed for computation. Additionally, we analyze the convergence of the R-GADI method and prove its consistency with the convergence of the GADI method. Finally, the effectiveness of the new algorithm is demonstrated through corresponding numerical experiments.}




\maketitle

\section{Introduction}\label{sec1}

This paper focuses on the numerical solution of large-scale continuous-time algebraic Riccati matrix equations (CARE):
\begin{equation}\label{eq5}
 A^{T}X+XA+Q-XGX=0,
\end{equation}
where $Q=C^{T}C$ is symmetric and positive definite, $G=BB^{T}$ is symmetric and positive semi-definite, and $A\in\mathbb R^{n\times n},\ B\in\mathbb R^{n\times m}, \ C\in\mathbb R^{p\times n}$ are known matrices, and $X\in\mathbb R^{n\times n}$ is unknown matrix. Here, rank$(C)$=$p$, rank$(B)$=$m$ and $p,\ m\ll n$.
The numerical treatment of this type of equation plays a significant role in various fields. For instance, linear quadratic regulators \cite{ref1}, linear model reduction systems based on equilibrium \cite{ref2}, parabolic partial differential equations and transport theory \cite{ref3,ref4}, Wiener-Hopf factorization of Markov chains \cite{ref5}, and factorization of rational matrix functions \cite{ref6}, etc. In this paper, we assume that the coefficient matrix $A$ is sparse. Typically, the stable solution to equation \eqref{eq5} is desired, where the solution $X$ is symmetric and positive semi-definite, and $A-GX$ is stable, meaning its eigenvalues have negative real parts. Such stable solutions exist and under certain assumptions, it is unique \cite{ref7}. When $n$ is large, it is common to seek low-rank approximations of the symmetric positive semi-definite solution in the form of $ZZ^{T}\approx X$, where the column rank of $Z$ is low, i.e., rank$(Z)\ll n$. Since storing the full matrix $X$ requires a significant amount of memory, considering only the storage of the matrix $Z$ allows us to optimize resource utilization when dealing with large-scale problems.

Firstly, we present an application of equation \eqref{eq5} by considering a linear time-invariant control system\cite{ref25}:
\begin{equation*}
\left\{
\begin{array}{ll}
\dot{x}(t)=Ax(t)+Bu(t),\ \ \ x(0)=x_{0},\\
y(t)=Cx(t),
\end{array}
\right.
\end{equation*}
where $x(t)\in\mathbb R^{n}$ is the state vector, $u(t)\in\mathbb R^{m}$ is the control vector, and $y(t)\in\mathbb R^{p}$ is the output vector. Quadratic optimal control aims to minimize
$$J(x_{0},u)=\frac{1}{2}\int_{0}^{+\infty}(y(t)^{T}y(t)+u(t)^{T}u(t))dt.$$
Assuming that $(A,B)$ is stabilizable, i.e., there exists a matrix $S$ such that $A-BS$ is stable. And $(C,A)$ is detectable, i.e., $(A^{T},C^{T})$ is stable, there exists a unique optimal solution $\bar{u}$ that minimizes the functional $J(x_{0},u)$\cite{ref8}, which can be determined through the feedback operator $P$, i.e., $\bar{u}(t)=Px(t)$, where $P=B^{T}X$, and $X\in\mathbb R^{n\times n}$ is the unique symmetric positive semi-definite stable solution of the matrix equation \eqref{eq5}.

There are many methods have been proposed for the numerical solution of equation \eqref{eq5}, including the Schur method \cite{ref9}, matrix sign function \cite{ref2,ref10}, structured doubling algorithm \cite{ref11}, symplectic Lanczos method \cite{ref12}, and projection methods based on the global Arnoldi process of the Krylov subspace \cite{ref13,ref14}. However, these methods often require multiple iterations to obtain an accurate approximate solution, leading to significant increases in computational time and memory requirements.\ To address this issue, some scholars have also investigated approximate low-rank solutions for computing large sparse matrix equations. Typically, we combine the Newton's iteration method with the alternating direction implicit (ADI) algorithm to solve such equations, and take advantage of the quadratic local convergence properties of Newton's method. However, at each Newton iteration step, solving a large Lyapunov matrix equation is required to obtain the next iteration solution. The continuous-time Lyapunov matrix equation \cite{refL.A.} as follows:
\begin{equation}\label{eq01}
 F^{T}X+XF=Q,
\end{equation}
where $Q=C^{T}C$, $F\in\mathbb R^{n\times n},\ C\in\mathbb R^{p\times n}$. From equation \eqref{eq5}, we observe that equation \eqref{eq01} is a specific case of equation \eqref{eq5} when $G=0$. If the spectrum of 
matrix $F$ is in the positive-real half-plane. Under these conditions $F$ and $-F^{T}$  have no common 
characteristic roots \cite{refRutherford}, and there is a unique symmetric solution X.

Benner et al. \cite{ref15} proposed the low-rank Kleinman-Newton ADI iteration method to solve equation \eqref{eq5}. In the computational process, the Lyapunov matrix equation is solved using a low-rank Cholesky factorization. This method is based on solving linear systems with shifted matrices $A+\alpha I$, where $\alpha$ is the ADI parameter. However, determining the optimal ADI parameter and finding approximations for the Lyapunov equation increase the burden on memory requirements and computational time. Recently, Wong and Balakrishnan \cite{ref16} introduced an algorithm called Quadratic ADI (qADI) method to solve the algebraic Riccati equation \eqref{eq5}. Their method is a direct extension of the Lyapunov ADI method. Additionally, Wong and Balakrishnan provided a low-level variant of this algorithm. However, this variant has a significant drawback: in every step, all low-rank factors need to be reconstructed, which greatly affects the performance of the algorithm. In addition to the qADI method, several approaches for solving large-scale Riccati equations have emerged in recent literature. For instance, Amodei and Buchot \cite{ref17} obtain approximate solutions by computing low-dimensional subspaces of the associated Hamiltonian matrix. Benner et al. \cite{ref18} propose a novel ADI iteration method called Riccati ADI (RADI), which expands each factor by several multiples of columns or rows while keeping the elements from previous steps unchanged. Their method yields low-rank Lyapunov ADI iteration formulas.

In addition, there are currently many algorithms available for solving the Lyapunov equation \eqref{eq01}. ADI iterations can significantly accelerate convergence if the optimal shifts of $A$ and $A^{T}$ can be effectively estimated. Therefore, for stable Lyapunov equations \eqref{eq01}, when solving large-scale sparse problems, ADI iteration methods are often preferred as they preserve sparsity and are more amenable to parallelization in most cases. Recent theoretical results \cite{ref19,ref20,ref21} indicate that using the Cholesky factorization-alternating direction implicit (CF-ADI) algorithm to compute low-rank approximations for the Lyapunov equation is effective. Based on the \cite{ref22}, the GADI iteration for solving large-scale sparse linear systems can be described as follows:
\begin{equation}\label{eq02}
Ax=b,\ \ A\in\mathbb R^{n\times n},\ \ x,b\in\mathbb R^{n}.
\end{equation}
\begin{itemize}
\item Firstly, the matrix $A$ is split, assuming that $A$ can be represented as $A=M+N $, and then assign parameters to obtain$$\alpha x+Mx=\alpha x-Nx+b,$$
$$\alpha x+Nx=Nx-(1-\omega)\alpha x+(1-\omega)\alpha x+\alpha x=Nx-(1-\omega)\alpha x+(2-\omega)\alpha x.$$

\item Next, the algorithm is obtained by alternating between these two splittings. Given an initial $x_{0}=0$, the GADI iteration computes a sequence ${x_{k}}$ as follows
\begin{equation}\label{eq03}
\left\{
\begin{array}{ll}
(\alpha I+M)x_{k+\frac{1}{2}}=(\alpha I-N)x_{k}+b,\\
(\alpha I+N)x_{k+1}=(N-(1-\omega)\alpha I)x_{k}+(2-\omega)\alpha x_{k+\frac{1}{2}},
\end{array}
\right.
\end{equation}
with $\alpha>0,\ \ 0\leq\omega<2$.

\item Specifically,  the matrix $A$ is split into $A=H+S$, where $H=\frac{A+A^{\ast}}{2}$ is a Hermiten matrix, and $S=\frac{A-A^{\ast}}{2}$ is a skew-Hermitian matrix, then GADI-HS format is obtained
\begin{equation}\label{eq04}
\left\{
\begin{array}{ll}
(\alpha I+H)x_{k+\frac{1}{2}}=(\alpha I-S)x_{k}+b,\\
(\alpha I+S)x_{k+1}=(S-(1-\omega)\alpha I)x_{k}+(2-\omega)\alpha x_{k+\frac{1}{2}}.
\end{array}
\right.
\end{equation}
\end{itemize}

In this paper, we propose a low-rank generalized alternating direction implicit iteration (R-GADI) algorithm, which is an improvement over the GADI algorithm for solving the Lyapunov equation. We represent the solution as a low-rank approximation $X\approx VW^{T}$, where rank$(V)$ and rank$(W)\ll n$. During the computation, the R-GADI method provides a low-rank approximation of the solution $X$, eliminating the need to store $X$ at each iteration and reducing storage requirements. Additionally, we combine the Kleinman-Newton method with R-GADI (referred to as Kleinman-Newton-RGADI) to solve the Riccati equation \eqref{eq5}.This method is a variant of the Newton-GADI algorithm \cite{ref23}, which significantly reduces the total number of ADI iterations and thus lowers the overall computational cost. Finally, numerical examples in the paper demonstrate the effectiveness of the proposed algorithm.

The remaining structure of this paper is as follows: in section \ref{sec:solve Lyapunov equation}, we introduce the R-GADI iteration format for solving the Lyapunov equation and demonstrates the consistency between R-GADI and GADI iterations in terms of convergence. The selection of parameters, algorithm complexity, and comparison with other methods are discussed, along with relevant numerical examples. In section \ref{sec:solve Riccati equation}, we first transform the Riccati equation into the Lyapunov equation using the Kleinman-Newton method, and then present the R-GADI iteration format. Convergence, algorithm complexity, and additional numerical examples are also discussed to validate the effectiveness of the proposed algorithm. Finally, in section \ref{sec:conclusion} concludes the paper by summarizing the findings and offering some concluding remarks.

In this article, we use the following notation: $\mathbb R^{n\times m}$ denotes the set of all $n\times m$ real matrices. If $A\in\mathbb R^{n\times n}$, then $A^{T}$ and $A^{-1}$ represent the transposition and inverse of $A$, respectively. The sets of eigenvalues and singular values of $A$ are denoted as $\Lambda(A)=\{\lambda_{i}(A),\ i=1,2,\cdots,n\}$ and $\Sigma(A)=\{\sigma_ {i}(A),\ i=1,2,\cdots,n\}$, where $\lambda_{i}(A)$ and $\sigma_ {i}(A)$ are the $i$-th eigenvalue and the $i$-th singular value of $A$, respectively. $\rho(A)=\max_{1\leq i\leq n}\{|\lambda_{i}(A)|\}$ represents the spectral radius of $A$. $A>0\ (A\geq0)$ indicates that $A$ is positive definite (positive semidefinite), $\|A\|_{2}$ denotes the 2-norm of $A$, Re$(A)$ and Im$(A)$ represent the real and imaginary parts of the eigenvalues of $A$, respectively. $A\otimes I$ denotes the Kronecker product of  $A$ and $I$. 
\begin{Definition}\label{definition1}
Let $A_{1}=[a_{ij}]\in\mathbb C^{m\times n},\ \ B_{1}\in\mathbb C^{p\times q}$, then
 \begin{equation*}
      A_{1}\otimes B_{1}=\left(
                                                                                                                 \begin{array}{cccc}
                                                                                                                   a_{11}B_{1} & a_{12}B_{1} & \cdots & a_{1n}B_{1} \\
                                                                                                                   a_{21}B_{1} & a_{22}B_{1} & \cdots & a_{2n}B_{1} \\
                                                                                                                   \vdots & \vdots &  & \vdots \\
                                                                                                                   a_{m1}B_{1} & a_{m2}B_{1} & \cdots & a_{mn}B_{1} \\
                                                                                                                 \end{array}
                                                                                                               \right)\in\mathbb C^{mp\times nq},
\end{equation*}
it's called Kronecker product of $A_ {1}$ and $B_{1}$.
\end{Definition}
\begin{Definition}\label{definition2}
If the vectorization operator vec satisfies $\mathbb C^{m\times n}\rightarrow\mathbb C^{mn}$:
$$vec(X_{1})=(x_{1}^{T},x_{2}^{T},\cdots,x_{n}^{T})^{T},\ \ X_{1}=[x_{1},x_{2},\cdots,x_{n}]\in\mathbb C^{m\times n},$$
then this operator is called a straightening operator.
\end{Definition}
\section{Low rank GADI for solving Lyapunov equation}
\label{sec:solve Lyapunov equation}
\subsection{Derivation of iterative format}
\qquad Firstly, we consider the ADI iterative method for solving the Lyapunov equation \eqref{eq01} with a single parameter.
\begin{equation}\label{eq06}
\left\{
\begin{array}{ll}
(F^{T}+\alpha I)X_{k+\frac{1}{2}}=Q-X_{k}(F-\alpha I),\\
X_{k+1}(F+\alpha I)=Q-(F^{T}-\alpha I)X_{k+\frac{1}{2}}.
\end{array}
\right.
\end{equation}
By \eqref{eq06}, we can obtain iterative format for $X_{k+1}$, $$X_{k+1}=(F^{T}-\alpha I)(F^{T}+\alpha I)^{-1}X_{k}(F-\alpha I)(F+\alpha I)^{-1}+2\alpha(F^{T}+\alpha I)^{-1}Q(F+\alpha I)^{-1}.$$

Since $F-\alpha I$ is interchangeable with $(F+\alpha I)^{-1}$, we can derive a low-rank ADI (R1-ADI) iterative formula with a single parameter.
\begin{equation}\label{eq07}
\left\{
\begin{array}{lllll}
V_{1}=\sqrt{2\alpha}(F^{T}+\alpha I)^{-1}C^{T} ,\ \ V_{1}\in\mathbb R^{n\times p},\\
V_{k}=[(F^{T}-\alpha I)(F^{T}+\alpha I)^{-1}V_{k-1},V_{1}] ,\ \ V_{k}\in\mathbb R^{n\times kp},\\
X_{k}=V_{k}V_{k}^{T},\ \ X_{k}\in R^{n\times n}.
\end{array}
\right.
\end{equation}

Next, we consider the ADI iterative method with two parameters for solving the Lyapunov equation \eqref{eq01}.
\begin{equation}\label{eq08}
\left\{
\begin{array}{ll}
(F^{T}+\alpha I)X_{k+\frac{1}{2}}=Q-X_{k}(F-\alpha I),\\
X_{k+1}(F+\beta I)=Q-(F^{T}-\beta I)X_{k+\frac{1}{2}}.
\end{array}
\right.
\end{equation}
Similarly, we can obtain a low-rank ADI (R2-ADI) iterative formula with two parameters.
\begin{equation}\label{eq09}
\left\{
\begin{array}{lllll}
V_{1}=\sqrt{\alpha+\beta}(F^{T}+\alpha I)^{-1}C^{T} ,\ \ V_{1}\in\mathbb R^{n\times p},\\
V_{k}=[(F^{T}-\beta I)(F^{T}+\alpha I)^{-1}V_{k-1},V_{1}] ,\ \ V_{k}\in\mathbb R^{n\times kp},\\
W_{1}=\sqrt{\alpha+\beta}(F^{T}+\beta I)^{-1}C^{T} ,\\
W_{k}=[(F^{T}+\beta I)^{-1}(F^{T}-\alpha I)W_{k-1},W_{1}] ,\\
X_{k}=V_{k}W_{k}^{T},\ \ X_{k}\in\mathbb R^{n\times n}.
\end{array}
\right.
\end{equation}

We apply the GADI iterative framework to solve the Lyapunov equation \eqref{eq01}. Firstly, the straightening operator is applied and from the Kronecker product, we have
\begin{equation}\label{eq10}
(F^{T}\otimes I+I\otimes F^{T})x=q,\ \ x=\text{vec}(X),\ \ q=\text{vec}(Q).
\end{equation}

Secondly, by applying the GADI iterative method in equation \eqref{eq10}, we obtain the following expression.
\begin{equation}\label{eq11}
\left\{
\begin{array}{ll}
(\alpha I_{n^{2}}+I\otimes F^{T})x_{k+\frac{1}{2}}=(\alpha I_{n^{2}}-F^{T}\otimes I)x_{k}+q,\\
(\alpha I_{n^{2}}+F^{T}\otimes I)x_{k+1}=(F^{T}\otimes I-(1-\omega)\alpha I_{n^{2}})x_{k}+(2-\omega)\alpha x_{k+\frac{1}{2}}.
\end{array}
\right.
\end{equation}
 We rewrite equation \eqref{eq11} into matrix form as
\begin{equation}\label{eq12}
\left\{
\begin{array}{ll}
(\alpha I+F^{T})X_{k+\frac{1}{2}}=X_{k}(\alpha I-F)+Q,\\
X_{k+1}(\alpha I+F)=X_{k}(F-(1-\omega)\alpha I)+(2-\omega)\alpha X_{k+\frac{1}{2}}.
\end{array}
\right.
\end{equation}
We select the appropriate parameter $\alpha$ to ensure that both matrices $\alpha I+F^{T}$ and $\alpha I+F$ are invertible. From the first equation of \eqref{eq12}, we can obtain $$X_{k+\frac{1}{2}}=(\alpha I+F^{T})^{-1}X_{k}(\alpha I-F)+(\alpha I+F^{T})^{-1}C^{T}C,$$  we substitute it into the second equation, then
\begin{equation*}
\begin{aligned}
X_{k+1}=&X_{k}(F-(1-\omega)\alpha I)(\alpha I+F)^{-1}\\
+&(2-\omega)\alpha X_{k+\frac{1}{2}}(\alpha I+F)^{-1}\\
=&X_{k}(F-(1-\omega)\alpha I)(\alpha I+F)^{-1}\\
+&(2-\omega)\alpha[(\alpha I+F^{T})^{-1}X_{k}(\alpha I-F)\\
+&(\alpha I+F^{T})^{-1}C^{T}C](\alpha I+F)^{-1}\\
=&X_{k}(F-(1-\omega)\alpha I)(\alpha I+F)^{-1}\\
+&(2-\omega)\alpha(\alpha I+F^{T})^{-1}X_{k}(\alpha I-F)(\alpha I+F)^{-1}\\
+&(2-\omega)\alpha(\alpha I+F^{T})^{-1}C^{T}C(\alpha I+F)^{-1}.
\end{aligned}
\end{equation*}
Taking the initial value $X_ {0}=0$, we have $X_{1}=(2-\omega)\alpha(\alpha I+F^{T})^{-1}C^{T}C(\alpha I+F)^{-1}$.

Let$$V_{1}=W_{1}=\sqrt{(2-\omega)\alpha}(\alpha I+F^{T})^{-1}C^{T},$$
where $X_{1}=V_{1}W_{1}^{T}$. We can also get
\begin{align*}
X_{2}&=X_{1}(F-(1-\omega)\alpha I)(\alpha I+F)^{-1}\\
&+(2-\omega)\alpha(\alpha I+F^{T})^{-1}X_{1}(\alpha I-F)(\alpha I+F)^{-1}\\
&+(2-\omega)\alpha(\alpha I+F^{T})^{-1}C^{T}C(\alpha I+F)^{-1}\\
&=V_{1}W_{1}^{T}(F-(1-\omega)\alpha I)(\alpha I+F)^{-1}\\
&+(2-\omega)\alpha(\alpha I+F^{T})^{-1}V_{1}W_{1}^{T}(\alpha I-F)(\alpha I+F)^{-1}+V_{1}W_{1}^{T}.
\end{align*}
Let$$V_{2}=[V_{1},\sqrt{(2-\omega)\alpha}(\alpha I+F^{T})^{-1}V_{1},V_{1}],$$
$$W_{2}=[(\alpha I+F^{T})^{-1}(F^{T}-(1-\omega)\alpha I)W_{1},\sqrt{(2-\omega)\alpha}(\alpha I+F^{T})^{-1}(\alpha I-F^{T})W_{1},W_{1}],$$where $X_{2}=V_{2}W_{2}^{T}$.

Based on the previous derivation method, we can obtain the R-GADI iterative format for soliving the Lyapunov equation \eqref{eq01}.
\begin{equation}\label{eq13}
\left\{
\begin{array}{lllll}
V_{1}=\sqrt{(2-\omega)\alpha}(\alpha I+F^{T})^{-1}C^{T} ,\\
V_{k}=[V_{k-1},\sqrt{(2-\omega)\alpha}(\alpha I+F^{T})^{-1}V_{k-1},V_{1}],\ \ V_{k}\in R^{n\times(2^{k}-1)p},\\
W_{1}=\sqrt{(2-\omega)\alpha}(\alpha I+F^{T})^{-1}C^{T} ,\\
W_{k}=[(\alpha I+F^{T})^{-1}(F^{T}-(1-\omega)\alpha I)W_{k-1},\sqrt{(2-\omega)\alpha}(\alpha I+F^{T})^{-1}(\alpha I-F^{T})W_{k-1},W_{1}] ,\\
X_{k}=V_{k}W_{k}^{T}.
\end{array}
\right.
\end{equation}
Next, the R-GADI algorithm is given as Algorithm \ref{algorithm1}:
\floatname{algorithm}{Algorithm}
\renewcommand{\algorithmicrequire}{\textbf{Input:}}
\renewcommand{\algorithmicensure}{\textbf{Output:}}
\begin{algorithm}
\caption{R-GADI iteration for solving Lyapunov equation \eqref{eq01}}
  \label{algorithm1}
  \begin{algorithmic}[1]
    \REQUIRE Matrix $F,\ C$, parameters $\alpha$ and $\omega$,\ $k_{max}$ and residual limit $\varepsilon$;
    \ENSURE Approximation $X\approx VW^{T}$ for the solution of the Lyapunov equation $F^{T} X+XF=Q$.
    \STATE $X_{0}=0;\ Q=C^{T}C$;
    \FOR {$k=1,\cdots,k_{max}$}
    \IF {$k=1$}
    \STATE Solve $\frac{1}{\sqrt{(2-\omega)\alpha}}(F^{T}+\alpha I)V_{1}=C^{T}$ for $V_{1};$
    \STATE $W_{1}=V_{1};$
    \ELSE
    \STATE $V_{k}=[V_{k-1},\ \sqrt{(2-\omega)\alpha}(F^{T}+\alpha I)^{-1}V_{k-1},\ \sqrt{(2-\omega)\alpha}(F^{T}+\alpha I)^{-1}C^{T}];$
    \STATE $W_{k}=[(\alpha I+F^{T})^{-1}(F^{T}-(1-\omega)\alpha I)W_{k-1},\ \sqrt{(2-\omega)\alpha}(\alpha I+F^{T})^{-1}(\alpha I-F^{T})W_{k-1},$
            \quad $\sqrt{(2-\omega)\alpha}(F^{T}+\alpha I)^{-1}C^{T}];$
    \ENDIF
    \STATE $X_{k}=V_{k}W_{k}^{T};$
    \STATE Compute Res$ (X_{k})=\frac{\|F^{T}X_{k}+X_{k}F-Q\|_{2}}{\|Q\|_{2}};$
               \IF {Res$(X_{k})<\varepsilon$}
               \STATE  stop;
               \ENDIF
    \ENDFOR
  \end{algorithmic}
\end{algorithm}
\subsection{Convergence analysis}
Some simple properties of Kronecker product can be easily derived from the Definition \ref{definition1} and Definition \ref{definition2}.
\begin{lemma}\label{lemma1}\cite{ref24}
Let $\alpha\in\mathbb C$,\ $A_{1}\in\mathbb C^{m\times n},\ \ B_{1}\in\mathbb C^{p\times q},\ \ X_{1}\in\mathbb C^{n\times p},\ \ C_{1}\in\mathbb C^{m\times n},\ \ D_{1}\in\mathbb C^{p\times q}$, then
\begin{itemize}
\item[(a)]$\alpha(A_{1}\otimes B_{1})=(\alpha A_{1})\otimes B_{1}=A_{1}\otimes(\alpha B_{1});$\\

\item[(b)]$(A_{1}\otimes B_{1})(C_{1}\otimes D_{1})=(A_{1}C_{1})\otimes(B_{1}D_{1});$\\

\item[(c)]$(A_{1}\otimes B_{1})^{T}=A_{1}^{T}\otimes B_{1}^{T};$\\

\item[(d)]$vec(A_{1}X_{1}B_{1})=(B_{1}^{T}\otimes A_{1})vec(X_{1});$\\

\item[(e)]$\Lambda(I_{p}\otimes A_{1}-B_{1}^{T}\otimes I_{n})=\{\lambda-\mu:\lambda\in\Lambda(A_{1}),\ \ \mu\in\Lambda(B_{1})\}.$
\end{itemize}
\end{lemma}
\begin{lemma}\label{lemma2}\cite{ref24}
Let $A\in\mathbb R^{n\times n}$,\ $\rho(A)=max\{|\lambda|:\lambda\in\Lambda(A)\}$ is the spectral radius of $A$, then for any compatible norm $\|\cdot\|$ on $\mathbb R^{n\times n}$, there is $\rho(A)\leq \|A\|$.
\end{lemma}
\begin{lemma}\label{lem03}\cite{ref24}
Let $\|\cdot\|$ be a matrix norm on $\mathbb{R}^{n\times n}$, and let $A,\ B \in \mathbb{R}^{n\times n}$. Then, it holds that $$\|A+B\|\leq\|A\|+\|B\|.$$
\end{lemma}
\begin{lemma}\label{lemma3}
Let $\mathscr{R}(X)=F^{T}X+XF-Q$, and $\{X_{k}\}$ be an approximate solution sequence of the Lyapunov equation \eqref{eq01} generated by GADI iteration \eqref{eq12}, and $X$ is a symmetric positive definite solution of equation \eqref{eq01}.\ Then for any $k\geq0$, we have\\
\begin{itemize}
\item[(1)]$(\alpha I+F^{T})(X_{k+\frac{1}{2}}-X)=(X_{k}-X)(\alpha I-F);$\\

\item[(2)]$(\alpha I+F^{T})(X_{k+\frac{1}{2}}-X_{k})=-\mathscr{R}(X_{k});$\\

\item[(3)]$\mathscr{R}(X_{k+\frac{1}{2}})=(X_{k+\frac{1}{2}}-X_{k})(F-\alpha I);$\\

\item[(4)]$(X_{k+1}-X)(\alpha I+F)=(X_{k}-X)F+(1-\omega)\alpha(X_{k+\frac{1}{2}}-X_{k})+(X_{k+\frac{1}{2}}-X);$\\

\item[(5)]$(X_{k+1}-X_{k+\frac{1}{2}})(\alpha I+F)=(X_{k+\frac{1}{2}}-X_{k})((1-\omega)\alpha I-F).$
\end{itemize}
\end{lemma}
\begin{proof}
\begin{itemize}
\item[(1)]From the first equation of \eqref{eq12}, we have $$(\alpha I+F^{T})(X_{k+\frac{1}{2}}-X)=X_{k}(\alpha I-F)+Q-(\alpha I+F^{T})X.$$  Since $Q-F^{T}X=XF$, then
\begin{align*}
(\alpha I+F^{T})(X_{k+\frac{1}{2}}-X)&=XF-X_{k}F+X_{k}\alpha-\alpha X\\
&=(X_{k}-X)(\alpha I-F).
\end{align*}
\item[(2)]
\begin{align*}
(\alpha I+F^{T})(X_{k+\frac{1}{2}}-X_{k})&=X_{k}(\alpha I-F)+Q-(\alpha I+F^{T})X_{k}\\
&=Q-X_{k}F-F^{T}X_{k}\\
&=-\mathscr{R}(X_{k}).
\end{align*}
\item[(3)]Due to $F^{T}X_{k+\frac{1}{2}}-Q=X_{k}(\alpha I-F)-\alpha X_{k+\frac{1}{2}},$ then
\begin{align*}
\mathscr{R}(X_{k+\frac{1}{2}})&=F^{T}X_{k+\frac{1}{2}}+X_{k+\frac{1}{2}}F-Q\\
&=X_{k}(\alpha I-F)-\alpha X_{k+\frac{1}{2}}+X_{k+\frac{1}{2}}F\\
&=(X_{k+\frac{1}{2}}-X_{k})(F-\alpha I).
\end{align*}
\item[(4)]From the second equation of \eqref{eq12}, we have
\begin{align*}
(X_{k+1}-X)(\alpha I+F)&=X_{k}(F-(1-\omega)\alpha I)+(2-\omega)\alpha X_{k+\frac{1}{2}}-X(\alpha I+F)\\
&=X_{k}F-X_{k}(1-\omega)\alpha I+(1-\omega)\alpha X_{k+\frac{1}{2}}+\alpha X_{k+\frac{1}{2}}-\alpha X-XF\\
&=(X_{k}-X)F+(1-\omega)\alpha(X_{k+\frac{1}{2}}-X_{k})+\alpha(X_{k+\frac{1}{2}}-X).
\end{align*}
\item[(5)]
\begin{align*}
(X_{k+1}-X_{k+\frac{1}{2}})(\alpha I+F)&=X_{k}(F-(1-\omega)\alpha I)+(2-\omega)\alpha X_{k+\frac{1}{2}}\\
&-X(\alpha I+F)-X_{k+\frac{1}{2}}(\alpha I+F)\\
&=X_{k}F-X_{k+\frac{1}{2}}F+(1-\omega)\alpha(X_{k+\frac{1}{2}}-X_{k})\\
&=(X_{k+\frac{1}{2}}-X_{k})((1-\omega)\alpha I-F).
\end{align*}
\end{itemize}
\end{proof}
From the Lemma \ref{lemma3} , we can prove that Theorem \ref{theorem1} holds.
\begin{theorem}\label{theorem1}
Let $X$ be the symmetric positive definite solution of Lyapunov equation \eqref{eq01}, and let the initial matrix $X_{0}=0$, and parameter $\alpha>0,\ 0\leq\omega<2$.\ Then the matrix sequence $\{X_{k}\}$ generated by the previous GADI iteration \eqref{eq12} holds the following inequality.
$$\|X_{k+1}-X\|_{2}\leq\delta(\alpha)\|X_{k}-X\|_{2}+\eta(\alpha,\omega)\|\mathscr{R}(X_{k})\|_{2},$$
where 
\begin{align*}
\delta(\alpha)=\|F(\alpha I+F)^{-1}\|_{2}+\alpha\|(\alpha I+F^{T})^{-1}\|_{2}\|(\alpha I-F)(\alpha I+F)^{-1}\|_{2},\\
\eta(\alpha,\omega)=|1-\omega|\alpha\|(\alpha I+F^{T})^{-1}\|_{2}\|(\alpha I+F)^{-1}\|_{2}.
\end{align*}
\end{theorem}
\begin{proof}
From Lemma \ref{lemma3}, we can conclude that
\begin{align*}
X_{k+1}-X&=[(X_{k}-X)F+(1-\omega)\alpha(X_{k+\frac{1}{2}}-X_{k})+\alpha(X_{k+\frac{1}{2}}-X)](\alpha I+F)^{-1}\\
&=(X_{k}-X)F(\alpha I+F)^{-1}-(1-\omega)\alpha(\alpha I+F^{T})^{-1}\mathscr{R}(X_{k})(\alpha I+F)^{-1}\\
&+\alpha(\alpha I+F^{T})^{-1}(X_{k}-X)(\alpha I-F)(\alpha I+F)^{-1}.
\end{align*}
Therefore, by Lemma \ref{lem03},we can get \begin{align*}
\|X_{k+1}-X\|_{2}&\leq\|(X_{k}-X)F(\alpha I+F)^{-1}\|_{2}+|1-\omega|\alpha\|(\alpha I+F^{T})^{-1}\mathscr{R}(X_{k})(\alpha I+F)^{-1}\|_{2}\\
&+\alpha\|(\alpha I+F^{T})^{-1}(X_{k}-X)(\alpha I-F)(\alpha I+F)^{-1}\|_{2}\\
&\leq\|(X_{k}-X)\|_{2}\|F(\alpha I+F)^{-1}\|_{2}+|1-\omega|\alpha\|(\alpha I+F^{T})^{-1}\|_{2}\|\mathscr{R}(X_{k})\|_{2}\|(\alpha I+F)^{-1}\|_{2}\\
&+\alpha\|(\alpha I+F^{T})^{-1}\|_{2}\|X_{k}-X\|_{2}\|(\alpha I-F)(\alpha I+F)^{-1}\|_{2}\\
&=[\|F(\alpha I+F)^{-1}\|_{2}+\alpha\|(\alpha I+F^{T})^{-1}\|_{2}\|(\alpha I-F)(\alpha I+F)^{-1}\|_{2}]\|X_{k}-X\|_{2}\\
&+|1-\omega|\alpha\|(\alpha I+F^{T})^{-1}\|_{2}\|(\alpha I+F)^{-1}\|_{2}\|\mathscr{R}(X_{k})\|_{2}.
\end{align*}
\end{proof}

\begin{theorem}\label{theorem2}
Let $X$ be the symmetric positive definite solution of Lyapunov equation \eqref{eq01}, and let parameters $\alpha>0$ and $0\leq\omega<2$, then for any $k=0,1,2, \cdots$,  the matrix $V_{k}W_{k}^{T}$ defined by \eqref{eq13} converges to $X$.  This is equivalent to the convergence of the iterative sequence ${X_ {k}}$ converging to $X$, where ${X_{k}}$ is obtained from \eqref{eq12}.
\end{theorem}
\begin{proof}
Here, we only need to prove the convergence of the iterative format \eqref{eq12}, since $\alpha>0$ and $A$ are stable, then $\alpha I_{n^{2}}+I\otimes F^{T}$ and $\alpha I_{n^{2}}+F^{T}\otimes I$ are non-singular. From the iteration framework \eqref{eq11}, we can obtain$$x_{k+\frac{1}{2}}=(\alpha I_{n^{2}}+I\otimes F^{T})^{-1}(\alpha I_{n^{2}}-F^{T}\otimes I)x_{k}+(\alpha I_{n^{2}}+I\otimes F^{T})^{-1}q,$$
\begin{align*}
x_{k+1}&=(\alpha I_{n^{2}}+F^{T}\otimes I)^{-1}[(F^{T}\otimes I-(1-\omega)\alpha I_{n^{2}})\\
&+(2-\omega)\alpha(\alpha I_{n^{2}}+I\otimes F^{T})^{-1}(\alpha I_{n^{2}}-F^{T}\otimes I)]x_{k}\\
&+(2-\omega)\alpha(\alpha I_{n^{2}}+F^{T}\otimes I)^{-1}(\alpha I_{n^{2}}+I\otimes F^{T})^{-1}q\\
&=(\alpha I_{n^{2}}+F^{T}\otimes I)^{-1}(\alpha I_{n^{2}}+I\otimes F^{T})^{-1}[(\alpha^{2}I_{n^{2}}\\
&+(I\otimes F^{T})(F^{T}\otimes I)-(1-\omega)\alpha(I\otimes F^{T}+F^{T}\otimes I)]x_{k}\\
&+(2-\omega)\alpha(\alpha I_{n^{2}}+F^{T}\otimes I)^{-1}(\alpha I_{n^{2}}+I\otimes F^{T})^{-1}q.
\end{align*}
Let $$M=I\otimes F^{T}+F^{T}\otimes I,$$
$$G(\alpha,\omega)=(2-\omega)\alpha(\alpha I_{n^{2}}+F^{T}\otimes I)^{-1}(\alpha I_{n^{2}}+I\otimes F^{T})^{-1},$$
$$T(\alpha,\omega)=(\alpha I_{n^{2}}+F^{T}\otimes I)^{-1}(\alpha I_{n^{2}}+I\otimes F^{T})^{-1}[(\alpha^{2}I_{n^{2}}+(I\otimes F^{T})(F^{T}\otimes I)-(1-\omega)\alpha M],$$
then
$$x_{k+1}=T(\alpha,\omega)x_{k}+G(\alpha,\omega)q.$$

Next, we need to prove that for $\alpha>0,\ 0\leq\omega<2$, there is $\rho(T(\alpha,\omega))<1$.
Since$$2\alpha M=-(\alpha I_{n^{2}}-I\otimes F^{T})(\alpha I_{n^{2}}-F^{T}\otimes I)+(\alpha I_{n^{2}}+I\otimes F^{T})(\alpha I_{n^{2}}+F^{T}\otimes I),$$then we can obtain
\begin{align*}
T(\alpha,\omega)&=(\alpha I_{n^{2}}+F^{T}\otimes I)^{-1}(\alpha I_{n^{2}}+I\otimes F^{T})^{-1}[(\alpha^{2}I_{n^{2}}+(I\otimes F^{T})(F^{T}\otimes I)-(1-\omega)\alpha M]\\
&=(\alpha I_{n^{2}}+F^{T}\otimes I)^{-1}(\alpha I_{n^{2}}+I\otimes F^{T})^{-1}[(\alpha I_{n^{2}}-I\otimes F^{T})(\alpha I_{n^{2}}-F^{T}\otimes I)+\omega\alpha M]\\
&=\frac{1}{2}(\alpha I_{n^{2}}+F^{T}\otimes I)^{-1}(\alpha I_{n^{2}}+I\otimes F^{T})^{-1}[2(\alpha I_{n^{2}}-I\otimes F^{T})(\alpha I_{n^{2}}-F^{T}\otimes I)+\omega2\alpha M]\\
&=\frac{1}{2}[(2-\omega)(\alpha I_{n^{2}}+F^{T}\otimes I)^{-1}(\alpha I_{n^{2}}+I\otimes F^{T})^{-1}(\alpha I_{n^{2}}-I\otimes F^{T})(\alpha I_{n^{2}}-F^{T}\otimes I)+\omega I_{n^{2}}]\\
&=\frac{1}{2}[(2-\omega)T(\alpha)+\omega I_{n^{2}}],
\end{align*}
where $$T(\alpha)=(\alpha I_{n^{2}}+F^{T}\otimes I)^{-1}(\alpha I_{n^{2}}+I\otimes F^{T})^{-1}(\alpha I_{n^{2}}-I\otimes F^{T})(\alpha I_{n^{2}}-F^{T}\otimes I).$$
Due to $\lambda(T(\alpha,\omega))=\frac{1}{2}[(2-\omega)\lambda(T(\alpha))+\omega]$, then we have $$\rho(T(\alpha,\omega))\leq\frac{1}{2}[(2-\omega)\rho(T(\alpha))+\omega].$$
Let $$\widetilde{T}(\alpha)=(\alpha I_{n^{2}}+I\otimes F^{T})^{-1}(\alpha I_{n^{2}}-I\otimes F^{T})(\alpha I_{n^{2}}-F^{T}\otimes I)(\alpha I_{n^{2}}+F^{T}\otimes I)^{-1},$$ it can be seen that $T(\alpha)$ is similar to $\widetilde{T}(\alpha)$ through the matrix $\alpha I_{n^{2}}+F^{T}\otimes I$. Therefore, by Lemma \ref{lemma2} we get
\begin{align*}
\rho(T(\alpha))&\leq\parallel(\alpha I_{n^{2}}+I\otimes F^{T})^{-1}(\alpha I_{n^{2}}-I\otimes F^{T})(\alpha I_{n^{2}}-F^{T}\otimes I)(\alpha I_{n^{2}}+F^{T}\otimes I)^{-1}\parallel_{2}\\
&\leq\parallel(\alpha I_{n^{2}}+I\otimes F^{T})^{-1}(\alpha I_{n^{2}}-I\otimes F^{T})\parallel_{2}\parallel(\alpha I_{n^{2}}-F^{T}\otimes I)(\alpha I_{n^{2}}+F^{T}\otimes I)^{-1}\parallel_{2}\\
&=\parallel F_{L}\parallel_{2}\parallel F_{R}\parallel_{2},
\end{align*}
where $$F_{L}=(\alpha I_{n^{2}}+I\otimes F^{T})^{-1}(\alpha I_{n^{2}}-I\otimes F^{T}),\ \ F_{R}=(\alpha I_{n^{2}}-F^{T}\otimes I)(\alpha I_{n^{2}}+F^{T}\otimes I)^{-1}.$$

For $x\in\mathbb R^{n^{2}\times1}$, we get
\begin{align*}
\parallel F_{L}\parallel_{2}^{2}&=\max_{\parallel x\parallel_{2}=1}\frac{\parallel(\alpha I_{n^{2}}-I\otimes F^{T})x\parallel_{2}^{2}}{\parallel(\alpha I_{n^{2}}+I\otimes F^{T})x\parallel_{2}^{2}}\\
&=\max_{\parallel x\parallel_{2}=1}\frac{\parallel(I\otimes F^{T})x\parallel_{2}^{2}-\alpha x^{T}(I\otimes F+I\otimes F^{T})x+\alpha^{2}}{\parallel(I\otimes F^{T})x\parallel_{2}^{2}+\alpha x^{T}(I\otimes F+I\otimes F^{T})x+\alpha^{2}}\\
&\leq\max_{\parallel x\parallel_{2}=1}\frac{\parallel(I\otimes F^{T})x\parallel_{2}^{2}-2\alpha\min Re(\lambda(I\otimes F))+\alpha^{2}}{\parallel(I\otimes F^{T})x\parallel_{2}^{2}+2\alpha\min Re(\lambda(I\otimes F))+\alpha^{2}}\\
&\leq\frac{\parallel I\otimes F^{T}\parallel_{2}^{2}-2\alpha\min Re(\lambda(I\otimes F))+\alpha^{2}}{\parallel I\otimes F^{T}\parallel_{2}^{2}+2\alpha\min Re(\lambda(I\otimes F))+\alpha^{2}}\\
&=\frac{\parallel I\otimes F^{T}\parallel_{2}^{2}-2\alpha\min Re(\lambda(F))+\alpha^{2}}{\parallel I\otimes F^{T}\parallel_{2}^{2}+2\alpha\min Re(\lambda(F))+\alpha^{2}}.
\end{align*}
Similarly, we have a conclusion
$$\parallel F_{R}\parallel_{2}^{2}\leq\frac{\parallel F^{T}\otimes I\parallel_{2}^{2}-2\alpha\min Re(\lambda(F))+\alpha^{2}}{\parallel F^{T}\otimes I\parallel_{2}^{2}+2\alpha\min Re(\lambda(F))+\alpha^{2}},$$
since $A$ is stable and $\alpha>0$, then $$\parallel F_{L}\parallel_{2}<1,\ \ \parallel F_{R}\parallel_{2}<1.$$
Therefore, $\rho(T(\alpha))<1$ and $\rho(T(\alpha,\omega))\leq\frac{1}{2}[(2-\omega)\rho(T(\alpha))+\omega]<1$.
\end{proof}

\subsection{Selection of parameters}
\qquad In this section, we first compare the convergence rates of the GADI method and the ADI method, where $\rho(T(\alpha,\ \omega))$ and $\rho(T(\alpha))$ used here are defined in the previous proof of Theorem \ref{theorem1}. Afterwards, we will provide a more practical method to select parameters.
\begin{theorem}\label{theorem3}
Assuming that the eigenvalues of matrix $F$ have positive real parts, define$$\rho(T(\alpha,\omega))=|\xi|=|a+bi|,\ \ \rho(T(\alpha))=|\eta|=|c+di|,$$
\begin{itemize}
\item[(i)]if $|\eta|^{2}\leq c$, then $$\rho(T(\alpha))<\rho(T(\alpha,\omega))<1;$$
\item[(ii)]if $|\eta|^{2}>c$ and $0<\omega<\frac{4(|\eta_{k}|^{2}-c_{k})}{(1-c_{k})^{2}+d_{k}^{2}}<2$, we have $$\rho(T(\alpha,\omega))<\rho(T(\alpha))<1.$$
\end{itemize}
\end{theorem}
\begin{theorem}\label{theorem4}
Let $\sigma_{j}(A)$ and $\lambda_{j}(F)$ represent the singular and eigenvalues of the coefficient matrix $F$ of the Lyapunov equation \eqref{eq01}, respectively, with $j=1,2,\cdots,n$. Let $$\mu=\max_{j}\sigma_{j}(F),\ \ \nu=\min_{j}Re(\lambda_{j}(F)),$$
then we can obtain the optimal parameters $\alpha^{\ast}$ as follows$$\alpha^{\ast}=arg \min\limits_{\alpha}\frac{\mu^{2}-2\alpha\nu+\alpha^{2}}{\mu^{2}+2\alpha\nu+\alpha^{2}}=\mu.$$
\end{theorem}
The proof of Theorem \ref{theorem3} and Theorem \ref{theorem4} is similar to the proof process of Theorem 2.7 and Theorem 2.9 in \cite{ref23}, and we will not delve into the details here.

For the parameter $\omega$, since $0\leq\omega<2$, special values 0 and 1 can be selected for validation first. The general method is to select the appropriate parameter $\omega$ by analyzing residuals in numerical examples.
\subsection{Complexity analysis}
\qquad According to the iterative scheme \eqref{eq11}, this method is primarily used for efficiently solving large sparse Lyapunov equations. The approach involves utilizing Cholesky factorization and representing the solution in the form of low-rank factors. Before starting Algorithm \ref{algorithm1}, we need to determine the values of parameters $\alpha$ and $\omega$. Finding suitable parameter values can be challenging as the selection of parameters significantly impacts the iterative results and convergence speed. In Algorithm \ref{algorithm1}, a stopping criterion needs to be set to compute the maximum number of iterations. One approach is to stop the iteration when the change in the approximate solution is small, i.e., $\parallel X_{k}-X_{k-1}\parallel<\varepsilon$, where
$$\parallel X_{k}-X_{k-1}\parallel=\parallel V_{k}W_{k}^{T}-V_{k-1}W_{k-1}^{T}\parallel.$$

However, computing the norm of the approximate solution on high-dimensional matrices can be computationally expensive. We adopt an alternative approach by measuring the relative residual of the approximate solution $X_{k}$ generated by the low-rank GADI iteration \eqref{eq12} at the $k$-th step. The relative residual $\text{Res}(X_{k})$ is defined as follows:
\begin{equation*}
\text{Res}(X_{k})=\frac{\|F^{T}X_{k}+X_{k}F-Q\|_{2}}{\|Q\|_{2}}.
\end{equation*}
Therefore, with a given accuracy requirement $\varepsilon$, we stop the iteration when Res$(X_{k})<\varepsilon$.

Next, we calculate the computational complexity of Algorithm \ref{algorithm1}. During the iteration process, the number of columns in matrices $V_{k}$ and $W_{k}$ increases with the number of iterations, i.e., $V_{k},\ W_{k}\in\mathbb R^{n\times(2^{k}-1)p}$, where $p\ll n$. Firstly, in the first iteration, the computational cost of obtaining $V_{1}$ is $2n^{2}p$, and at this point, $W_{1}=V_{1}$. Then, in the $k$-th iteration, the computational cost of obtaining $V_{k}$ is $2n^{2}(2^{k-1}-1)p$, and the computational cost of obtaining $W_{k}$ is $4n^{3}+2n^{2}(2^{k-1}-1)p+2n^{2}p$. Therefore, the total computational complexity of this algorithm is $4n^{3}+4n^{2}2^{k-1}p$.
\subsection{Numerical experiments}\

\textbf{Example 2.5.1\ \ }\ We take the coefficient matrix of the Lyapunov equation \eqref{eq01} as$$F=
\left(
  \begin{array}{ccccc}
    5   &  0.3   &  0   &  \cdots   &  0\\
     0.2  &  5   &  0.3  &   \cdots   &  0\\
     \vdots   &  \ddots  &  \ddots  &   \ddots   &  \vdots\\
     0   &  \cdots   &  0.2  &   5   &  0.3\\
     0   & \cdots   &  0  &   0.2  &   5\\
  \end{array}
\right)_{n\times n},$$
$$C=\left(
                     \begin{array}{ccccc}
                       1, & 1 ,& \cdots ,& 1 ,& 1 \\
                     \end{array}
                   \right)_{1\times n},$$
with $Q=C^{T}C$.

We set the parameters $\omega=0.015$ and $\alpha$ to be $\sqrt{\lambda_{max}(F)\lambda_{min}(F)}$ and $\max\sigma(F)$, respectively. We compare the numerical results for different matrix dimensions and different choices of $\alpha$, which are listed in Table \ref{table1} below. From the data in the table, it can be observed that during the iteration process, when the matrix dimensions are the same and $\alpha$ is chosen as the maximum singular value of matrix $F$, the total iteration time is shorter, resulting in better performance.
\begin{table}[!htbp]
	\centering
	\begin{tabular}{cccccc}
	\hline
$n$ & algorithm & $\alpha$ & Res &IT &CPU\\ \hline
128& R-GADI & $\sqrt{\lambda_{max}(F)\lambda_{min}(F)}$  & 4.2129e-16 & 8 & 0.06s  \\
128&R-GADI &$\max\sigma(F)$  & 4.5781e-16 & 8 & 0.05s  \\ \hline
256& R-GADI & $\sqrt{\lambda_{max}(F)\lambda_{min}(F)}$  & 4.3544e-16 & 8 & 0.24s  \\
256&R-GADI& $\max\sigma(F)$  & 4.3033e-16 & 8 & 0.22s  \\ \hline
512&R-GADI &  $\sqrt{\lambda_{max}(F)\lambda_{min}(F)}$  & 2.3044e-16 & 8 & 1.34s  \\
 512&R-GADI &$\max\sigma(F)$  & 5.7213e-16 & 7 & 1.27s  \\ \hline
1024& R-GADI & $\sqrt{\lambda_{max}(F)\lambda_{min}(F)}$  & 4.3802e-16 & 8 & 28.64s  \\
1024& R-GADI&$\max\sigma(F)$  & 9.9827e-16 & 7 & 15.51s  \\ \hline
2048& R-GADI & $\sqrt{\lambda_{max}(F)\lambda_{min}(F)}$  & 9.4348e-15 & 7 &219.72s  \\
 2048&R-GADI &$\max\sigma(F)$  & 1.1144e-15 & 7 & 191.26s  \\  \hline
4096& R-GADI & $\sqrt{\lambda_{max}(F)\lambda_{min}(F)}$  & 4.4282e-16 & 8 &4408.41s  \\
4096& R-GADI& $\max\sigma(F)$  & 8.887e-16 & 7 & 1223.03s  \\ \hline
 \end{tabular}
 \caption{Numerical results for Example 2.5.1\label{table1}}
  \end{table}

Next, we set the parameters $\alpha=\max\sigma(F),\ \omega=0.015$. In the case of matrix dimensions being multiplied, we solve the problem using both the GADI method and the R-GADI method. Table \ref{table2} presents the numerical results for relative residual, iteration count, and CPU time. In terms of iteration count and time, the R-GADI method is relatively more efficient. For instance, when $n=1024$, Fig. \ref{figure1} shows the iteration count and relative residual obtained by applying these two methods iteratively, while Fig. \ref{figure2} displays the computation time required by each method. It is evident from the figures that the R-GADI method is effective.

\begin{table}[!htbp]
	\label{table2}
	\centering
	\begin{tabular}{ccccc}
	\hline
$n$ &algorithm & Res &IT &CPU\\ \hline
128& GADI & 9.5268e-16 & 8 & 0.07s  \\
128& R-GADI  & 4.5781e-16 & 8 & 0.05s  \\ \hline
256& GADI & 2.1253e-16 & 8 & 0.25s  \\
256& R-GADI  & 4.3033e-16 & 8 & 0.22s  \\ \hline
512& GADI & 4.2985e-16 & 8 & 1.49s  \\
512 & R-GADI  & 5.7213e-16 & 7 & 1.27s  \\ \hline
1024& GADI & 5.4699e-16 & 7 & 30.95s  \\
1024 & R-GADI  & 9.9827e-16 & 7 & 15.51s \\ \hline
2048& GADI & 9.8707e-15 & 7 & 380.3s  \\
2048& R-GADI  & 1.1144e-15 & 7 & 191.26s \\ \hline
4096& GADI & 9.6536e-15 & 7 & 2568.69s  \\ 
4096& R-GADI  & 8.887e-16 & 7 & 1223.03s  \\ \hline
\hline
	\end{tabular}
	\caption{Numerical results for Example 2.5.1\label{table2}}
	\end{table}

 \begin{figure}[H]\label{figure1}
\centering
    \begin{minipage}[t]{0.49\textwidth}
        \centering
        \includegraphics[width=1.1\textwidth]{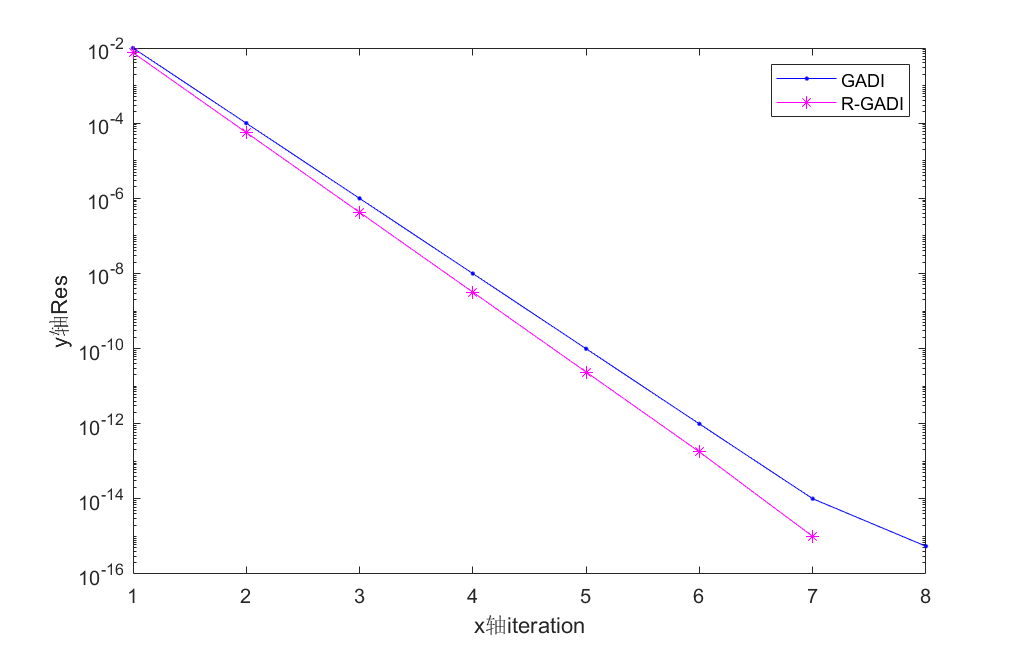}
        \caption{The residual curve of Example 2.5.1~$n=1024$\label{figure1}}
    \end{minipage}
    \begin{minipage}[t]{0.49\textwidth}
        \centering
        \includegraphics[width=1.1\textwidth]{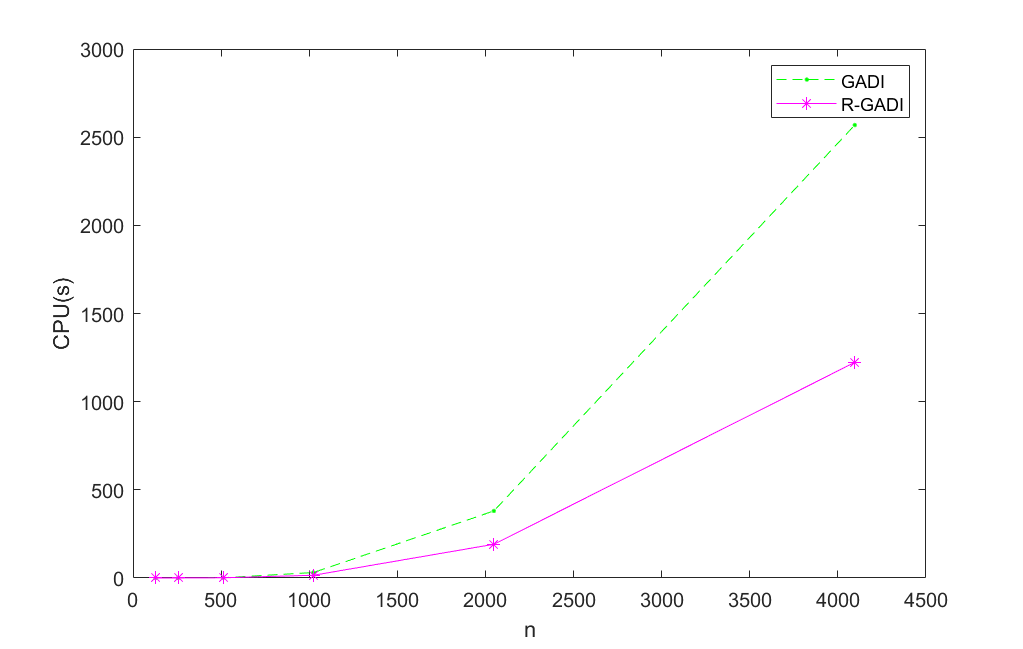}
        \caption{The time curve of Example 2.5.1\label{figure2}}
        \end{minipage}
\end{figure}
\textbf{Example 2.5.2\ \ }\ We take the coefficient matrix of the Lyapunov equation \eqref{eq01} as$$F=
\left(
  \begin{array}{ccccc}
    9   &  3   &  0   &  \cdots   &  0\\
     -2  &  9   &  3  &   \cdots   &  0\\
     \vdots   &  \ddots  &  \ddots  &   \ddots   &  \vdots\\
     0   &  \cdots   &  -2  &   9   &  3\\
     0   & \cdots   &  0  &   -2  &   9\\
  \end{array}
\right)_{n\times n}, $$
$$C=\left(
                     \begin{array}{ccccc}
                       1, & 1 ,& \cdots ,& 1 ,& 1 \\
                     \end{array}
                   \right)_{1\times n}.$$

We utilize the GADI, R1-ADI, R2-ADI, and R-GADI methods to solve the problem. Table 3 presents the relative residual, iteration count, and time obtained using these four iterative methods as the matrix dimension increases. By comparing these data, we can observe the effectiveness of the R-GADI method in solving the problem. Additionally, Fig. \ref{figure3} illustrates the correspondence between iteration count and relative residual for these four iterative methods when $n=1024$. Fig. \ref{figure4} displays the iteration time required by each method, it is evident that the R-GADI method achieves better results in solving the problem.

\begin{table}[htb]
\centering
\begin{tabular}{ccccc}
   \hline
   $n$ &algorithm & Res &IT &CPU\\ \hline
   128& GADI & 2.1884e-16 & 14 & 0.15s  \\
   128& R1-ADI  & 4.8137e-16 & 14 & 0.11s  \\
  128 & R2-ADI  & 3.0187e-16 & 12 & 0.12s  \\
   128& R-GADI  & 6.2135e-16 & 10 & 0.07s  \\ \hline
    256& GADI & 8.0851e-16 & 13 & 0.65s  \\
   256& R1-ADI  & 2.5974e-16 & 14 & 0.47s  \\
   256& R2-ADI  & 7.9983e-16 & 11 & 0.43s  \\
  256 & R-GADI  & 3.1158e-16 & 10 & 0.29s  \\ \hline
    512& GADI & 4.4317e-16 & 12 & 4.22s  \\
   512& R1-ADI  & 2.6518e-16 & 13 & 2.29s  \\
   512& R2-ADI  & 4.2145e-16 & 11 & 2.27s  \\
   512& R-GADI  & 2.0157e-16 & 10 & 1.48s  \\ \hline
    1024& GADI & 1.7196e-16  & 12 & 45.34s  \\
   1024& R1-ADI  & 3.5374e-16 & 12 & 26.63s  \\
   1024& R2-ADI  & 8.1825e-16 & 10 & 26.86s  \\
   1024& R-GADI  & 2.2622e-16 & 10 & 15.21s  \\ \hline
    2048& GADI & 6.0359e-16 & 12 & 490.81s  \\
   2048& R1-ADI  & 1.2585e-15 & 12 & 455.37s  \\
   2048& R2-ADI  & 1.0059e-15 & 10 & 450.4s  \\
   2048& R-GADI  & 6.6674e-16 & 9 & 234.92s  \\ \hline
    4096& GADI & 2.8412e-16 & 12 & 3998.2s  \\
   4096& R1-ADI  & 3.2722e-16 & 12 & 3225.4s  \\
   4096& R2-ADI  & 4.3943e-15 & 9 & 3025.7s  \\
   4096& R-GADI  & 2.983e-16 & 9 & 1495.60s  \\ \hline
 \end{tabular}
 \caption{Numerical results for Example 2.5.2\label{table3}}
 \end{table}
 \begin{figure}[H]
\centering
    \begin{minipage}[t]{0.49\textwidth}
        \centering
        \includegraphics[width=1.1\textwidth]{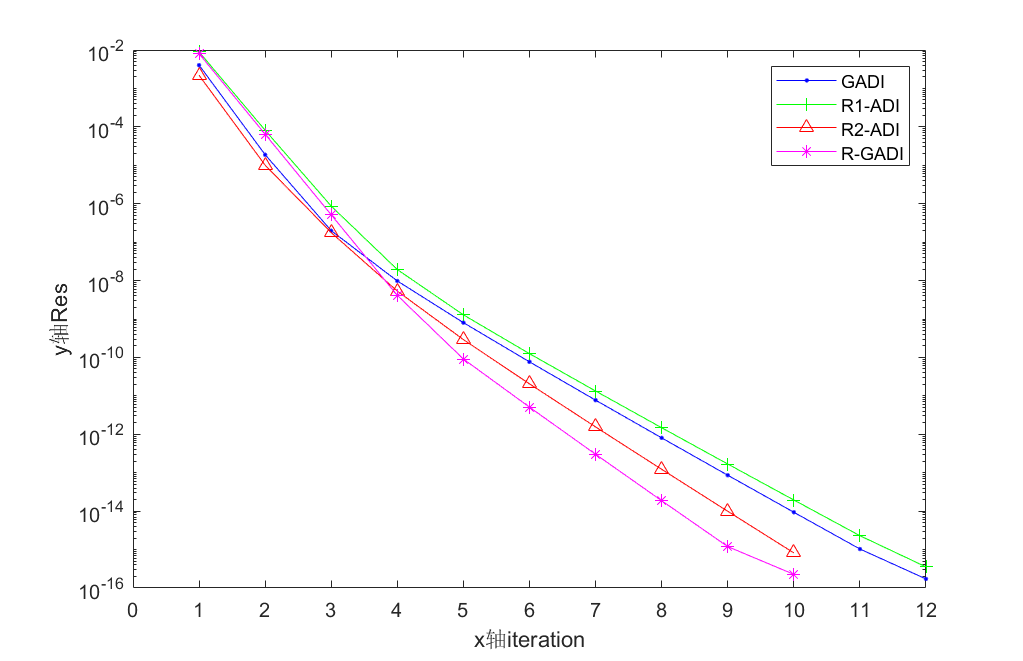}
        \caption{The residual curve of Example 2.5.2\label{figure3}~$n=1024$}
    \end{minipage}
    \begin{minipage}[t]{0.49\textwidth}
        \centering
        \includegraphics[width=1.1\textwidth]{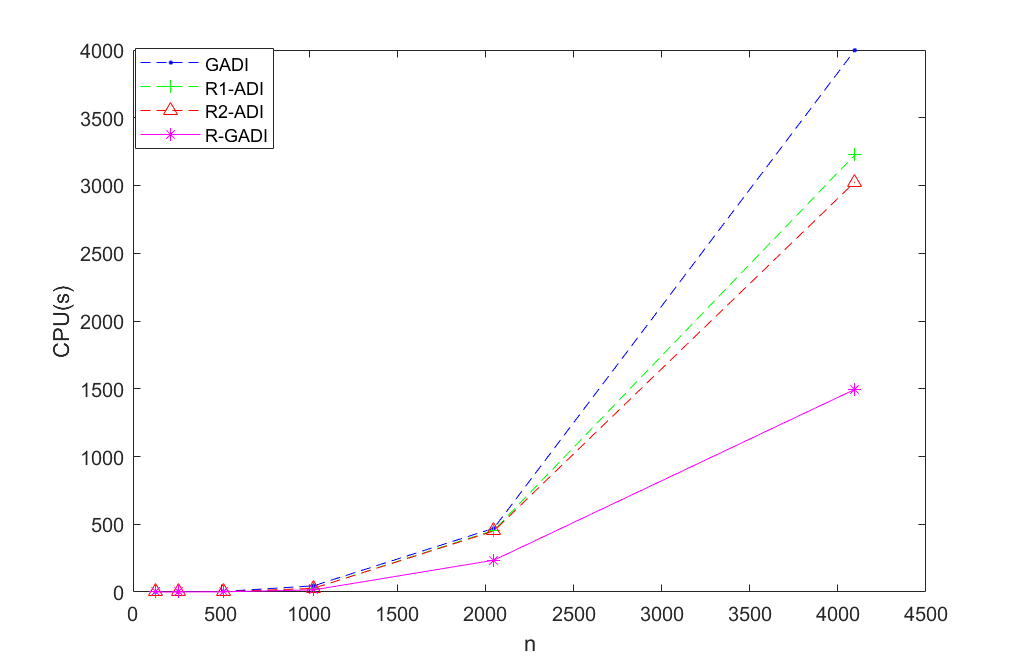}
        \caption{The time curve of Example 2.5.2\label{figure4}}
        \end{minipage}
\end{figure}
From the above Fig. \ref{figure3}, we can see that in the initial iteration phase, these four iterative methods exhibit relatively fast convergence. However, as the iterations progress, the convergence speed slows down. Upon reaching a certain number of iterations, it becomes evident that the R-GADI method demonstrates better convergence speed and requires the least amount of time.
\section{Low rank GADI for solving Riccati equation}
\label{sec:solve Riccati equation}
\subsection{Derivation of iterative format}
In this section, we investigate the conclusions related to solving the continuous algebraic Riccati equation \eqref{eq5} using the GADI method. The algebraic Riccati equation has its origins in numerical problems in control theory and finds wide applications, particularly in the design of quadratic optimal control. For certain linear quadratic optimization control problems, the problem can eventually be transformed into solving the algebraic Riccati equation for a stable solution. Generally, the solution of the algebraic Riccati equation \eqref{eq5} is not unique. Therefore, we first provide a sufficient condition for the existence of a unique solution.
\begin{lemma}\label{lemma4}\cite{ref24}
If the matrix pair $(A, B)$ in the algebraic Riccati equation \eqref{eq5} is stabilizable, meaning that for any $\lambda$ such that Re$\lambda\geq0$, the matrix $[A-\lambda I\ B]$ has full row rank, and if the matrix pair $(C,\ A)$ is detectable, meaning that for any $\lambda$ and $x$ such that Re$\lambda\geq0$ and $Ax=\lambda x$, we have $Cx\neq0$, then equation \eqref{eq5} has a unique positive semi-definite solution, and this solution is stable.
\end{lemma}
Firstly, we consider the Newton iteration method:
\begin{itemize}
\item Define mapping $F:\mathbb R^{n\times n}\rightarrow\mathbb R^{n\times n}$ is $$F(X)=A^{T}X+XA+Q-XGX, \ \ X\in\mathbb R^{n\times n},$$ then we have
\begin{align*}
F(X+E)&=A^{T}(X+E)+(X+E)A+Q-(X+E)G(X+E)\\
&=A^{T}X+XA-XGX+Q+A^{T}E+EA-EGX-XGE-EGE\\
&=F(X)+L(X)+O(\|E\|_{2}^{2}) \ (E\rightarrow0),
\end{align*}
where
\begin{align*}
L(E)&=A^{T}E+EA-EGX-XGE\\
&=(A-GX)^{T}E+E(A-GX),
\end{align*}
it is a linear operator. From the definition of Fr\'{e}chet differentiability, it follows that for any $X\in\mathbb R^{n\times n},\ F$ is Fr\'{e}chet differentiable at $X$. There is $$F^{\prime}(X)(E)=(A-GX)^{\ast}E+E(A-GX),\ E\in\mathbb R^{n\times n}.$$
\end{itemize}
Thus, by applying Newton's method $F^{\prime}(X_{k})(E_{k})=-F(X_{k})$, we can get the iterative format of the Newton iteration method for solving equation \eqref{eq5}.
\begin{itemize}
\item[(a)] Seeking $E_{k}$ that satisfies the equation
\begin{equation}\label{eq14}
  (A-GX_{k})^{T}E_{k}+E_{k}(A-GX_{k})=-F(X_{k}).
\end{equation}
\item[(b)] Calculate $X_{k+1}=X_{k}+E_{k}$, where $X_{0}\in\mathbb R^{n\times n}$ is the initial matrix. Rearranging equation \eqref{eq14} and using $X_{k+1}=X_{k}+E_{k}$, then we can obtain the equation satisfied by $X_{k+1}$
\end{itemize}
\begin{equation}\label{eq15}
  (A-GX_{k})^{T}X_{k+1}+X_{k+1}(A-GX_{k})+X_{k}GX_{k}+Q=0.
\end{equation}
\begin{itemize}
\item Next, we apply the Kleinman-Newton method\cite{ref25} with an initial feedback matrix $K_{0}=0$. Let $K_{k}=X_{k}B$ to obtain another equivalent form of equation \eqref{eq15}.
\begin{equation}\label{eq016}
 (A-BK_{k}^{T})^{T}X_{k+1}+X_{k+1}(A-BK_{k}^{T})=-K_{k}K_{k}^{T}-C^{T}C,
\end{equation}
with $A_{k}=BK_{k}^{T}-A,\ M_{k}=[K_{k},\ C^{T}].$ Therefore, equation \eqref{eq016} can be rewritten as the following Lyapunov equation form
\begin{equation}\label{eq17}
 A_{k}^{T}X_{k+1}+X_{k+1}A_{k}=M_{k}M_{k}^{T}.
\end{equation}
\end{itemize}
\begin{itemize}
\item In addition, we apply the GADI framework to solve the Lyapunov equation \eqref{eq17} and obtain an iterative format
\begin{equation}\label{eq18}
\left\{
\begin{array}{ll}
(\alpha_{k}I+A_{k}^{T})X_{k+1}^{(l+\frac{1}{2})}=X_{k+1}^{(l)}(\alpha_{k}I-A_{k})+Q_{k},\\
X_{k+1}^{(l+1)}(\alpha_{k}I+A_{k})=X_{k+1}^{(l)}(A_{k}-(1-\omega_{k})\alpha_{k}I)+(2-\omega_{k})\alpha_{k}X_{k+1}^{(l+\frac{1}{2})},
\end{array}
\right.
\end{equation}
where $$Q_{k}=M_{k}M_{k}^{T},\ \ l=0,1,\cdots,s,$$
$$X_{k+1}^{(0)}=X_{k},\ \ X_{k+1}=X_{k+1}^{(s)},\ \ \alpha_{k}>0,\ \ 0\leq\omega_{k}<2.$$
\end{itemize}
Next, we further manipulate the iterative scheme \eqref{eq18} and obtain the following expression under the assumption that $\alpha_{k}I+A_{k}^{T}$ and $\alpha_{k}I+A_{k}$ are non-singular.
\begin{align*}
X_{k+1}^{(l+1)}&=X_{k+1}^{(l)}(A_{k}-(1-\omega_{k})\alpha_{k}I)(\alpha_{k}I+A_{k})^{-1}\\
&+(2-\omega_{k})\alpha_{k}(\alpha_{k}I+A_{k}^{T})^{-1}X_{k+1}^{(l)}(\alpha_{k}I-A_{k})(\alpha_{k}I+A_{k})^{-1}\\
&+(2-\omega_{k})\alpha_{k}(\alpha_{k}I+A_{k}^{T})^{-1}Q_{k}(\alpha_{k}I+A_{k})^{-1}.
\end{align*}
Due to $X_{k+1}^{(0)}=X_{k}$, then we have
\begin{align*}
X_{k+1}^{(1)}&=X_{k}(A_{k}-(1-\omega_{k})\alpha_{k}I)(\alpha_{k}I+A_{k})^{-1}\\
&+(2-\omega_{k})\alpha_{k}(\alpha_{k}I+A_{k}^{T})^{-1}X_{k}(\alpha_{k}I-A_{k})(\alpha_{k}I+A_{k})^{-1}\\
&+(2-\omega_{k})\alpha_{k}(\alpha_{k}I+A_{k}^{T})^{-1}Q_{k}(\alpha_{k}I+A_{k})^{-1}.
\end{align*}
Let the initial value of iteration $X_ {0}=0 $, we get
$$X_{1}^{(1)}=(2-\omega_{0})\alpha_{0}(\alpha_{0}I+A_{0}^{T})^{-1}Q_{0}(\alpha_{0}I+A_{0})^{-1}=V_{1}^{(1)}(W_{1}^{(1)})^{T},$$
where \begin{align*}
V_{1}^{(1)}&=\sqrt{(2-\omega_{0})\alpha_{0}}(\alpha_{0}I+A_{0}^{T})^{-1}M_{0}\\
&=\sqrt{(2-\omega_{0})\alpha_{0}}(\alpha_{0}I-A^{T})^{-1}C^{T},\\
&W_{1}^{(1)}=V_{1}^{(1)}.
\end{align*}
Therefore, we provide a low rank Kleinman Newton GADI iterative scheme for the corresponding Riccati equation.
\begin{equation}\label{eq19}
\left\{
\begin{array}{lllllllllll}
A_{k}=BK_{k}^{T}-A,\ M_{k}=[K_{k},\ C^{T}],\ \ Q_{k}=M_{k}M_{k}^{T},\\
V_{1}^{(1)}=\sqrt{(2-\omega_{0})\alpha_{0}}(\alpha_{0}I-A^{T})^{-1}C^{T},\\
V_{k+1}^{(1)}=\sqrt{(2-\omega_{k})\alpha_{k}}(\alpha_{k}I+A_{k}^{T})^{-1}M_{k},\\
V_{k+1}^{(l)}=[V_{k+1}^{(l-1)},\sqrt{(2-\omega_{k})\alpha_{k}}(\alpha_{k}I+A_{k}^{T})^{-1}V_{k+1}^{(l-1)},V_{k+1}^{(1)}],\\
W_{1}^{(1)}=V_{1}^{(1)},\\
W_{k+1}^{(1)}=\sqrt{(2-\omega_{k})\alpha_{k}}(\alpha_{k}I+A_{k}^{T})^{-1}M_{k},\\
W_{k+1}^{(l)}=[(\alpha_{k}I+A_{k}^{T})^{-1}(A_{k}^{T}-(1-\omega_{k})\alpha_{k}I)W_{k+1}^{(l-1)},\\
\ \ \ \ \ \ \ \ \ \sqrt{(2-\omega_{k})\alpha_{k}}(\alpha_{k}I+A_{k}^{T})^{-1}(\alpha_{k}I-A_{k}^{T})W_{k+1}^{(l-1)},W_{k+1}^{(1)}],\\
X_{k+1}^{(l)}=V_{k+1}^{(l)}(W_{k+1}^{(l)})^{T},\\
K_{k+1}= V_{k+1}W_{k+1}^{T}B.
\end{array}
\right.
\end{equation}
The Algorithm \Ref{algorithm2} process is as follows:
\floatname{algorithm}{Algorithm}
\renewcommand{\algorithmicrequire}{\textbf{Input:}}
\renewcommand{\algorithmicensure}{\textbf{Output:}}
\begin{algorithm}
  \caption{Low rank GADI iteration for solving Lyapunov equation \eqref{eq17}}
  \label{algorithm2}
  \begin{algorithmic}[1]
    \REQUIRE Matrix $A,\ B,\ C$, initial feedback matrix $K_{0}$ makes $BK_{0}^{T}-A$ is stable, $k_{max}$;
    \ENSURE  Approximation $X_{k+1}\approx V_{k+1}W_{k+1}^{T}$ for the solution of the equation $A_{k}^{T}X_{k+1}+X_{k+1}A_{k}=Q_{k}$.
    \FOR {$k=1,\cdots,k_{max}$}
    \STATE $A_{k}=BK_{k}^{T}-A;$
    \STATE $M_{k}=[K_{k},\ C^{T}];\ Q_{k}=M_{k}M_{k}^{T};$
    \STATE $V_{k}^{(1)}=\sqrt{(2-\omega_{k-1})\alpha_{k-1}}(\alpha_{k-1}I+A_{k-1}^{T})^{-1}M_{k-1};$
    \STATE $W_{k}^{(1)}=\sqrt{(2-\omega_{k-1})\alpha_{k-1}}(\alpha_{k-1}I+A_{k-1}^{T})^{-1}M_{k-1};$
           \FOR {$l=2,3,\cdots, i_{k}^{max}$}
                \STATE $V_{k}^{(l)}=[V_{k}^{(l-1)},\ \ \sqrt{(2-\omega_{k-1})\alpha_{k-1}}(\alpha_{k-1}I+A_{k-1}^{T})^{-1}V_{k}^{(l-1)},\ \ V_{k}^{(1)}];$
                \STATE $W_{k}^{(l)}=[(\alpha_{k-1}I+A_{k-1}^{T})^{-1}(A_{k-1}^{T}-(1-\omega_{k-1})\alpha_{k-1}I)W_{k}^{(l-1)},\ $ $\qquad\sqrt{(2-\omega_{k-1})\alpha_{k-1}}(\alpha_{k-1}I+A_{k-1}^{T})^{-1}(\alpha_{k-1}I-A_{k-1}^{T})W_{k}^{(l-1)},\quad W_{k}^{(1)}];$
                \STATE $X_{k}^{(l)}=V_{k}^{(l)}(W_{k}^{(l)})^{T};$
           \ENDFOR
           \STATE $V_{k+1}=V_{k+1}^{(i_{k}^{max})},\ W_{k+1}=W_{k+1}^{(i_{k}^{max})}$;
           \STATE $K_{k+1}= V_{k+1}W_{k+1}^{T}B$;
    \ENDFOR
   \end{algorithmic}
\end{algorithm}
\floatname{algorithm}{Algorithm}
\renewcommand{\algorithmicrequire}{\textbf{Input:}}
\renewcommand{\algorithmicensure}{\textbf{Output:}}
\begin{algorithm} 
  \caption{A low rank Kleinman-Newton GADI method for solving Riccati equation \eqref{eq5}}
\label{algorithm3}
  \begin{algorithmic}[1]
    \REQUIRE Matrix $A,\ B,\ C$, initial feedback matrix $K_{0}$, loop variables $k=1,\cdots,i_{k}^{max}$, and relative residual precision $\varepsilon$;
    \ENSURE $X=X_{k+1}^{(i_{k}^{max})}$ make $V_{k+1}^{(i_{k}^{max})}(W_{k+1}^{(i_{k}^{max})})^{T}\approx X_{k+1}^{(i_{k}^{max})}$, where $X$ is the solution of equation $A^{T}X+XA-XBB^{T}X+C^{T}C=0$.
    \FOR {$k=1,\cdots,i_{k}^{max}$}
    \STATE $A_{k}=BK_{k}^{T}-A;$
    \STATE $M_{k}=[K_{k},\ C^{T}];\ Q_{k}=M_{k}M_{k}^{T};$
    \STATE Using Algorithm \ref{algorithm2} for low rank GADI to find $X_{k+1}\approx V_{k+1}W_{k+1}^{T}$ is the approximate solution of the Lyapunov equation $A_{k}^{T}X_{k+1}+X_{k+1}A_{k}=Q_{k}$.
    \STATE $K_{k+1}=V_{k+1}W_{k+1}^{T}B;$
    \STATE Compute Res$(V_{k+1}W_{k+1}^{T})=\frac{\|A_{k}^{T}V_{k+1}W_{k+1}^{T}+V_{k+1}W_{k+1}^{T}A_{k}-Q_{k}\|_{2}}{\|Q_{k}\|_{2}};$
               \IF {Res$(V_{k+1}W_{k+1}^{T})<\varepsilon$}
               \STATE  stop;
               \ENDIF
    \ENDFOR
  \end{algorithmic}
\end{algorithm}
\subsection{Convergence analysis}
Next, we provide the convergence results of using the Newton method to solve equation \eqref{eq15}.

\begin{theorem}\label{theorem5}\cite{ref24}
 For the algebraic Riccati equation \eqref{eq5}, assuming $(A,\ G)$ is stable and $(Q,\ A)$ is detectable, and selecting a symmetric positive semi-definite matrix $X_ {0}$ such the $A-GX_ {0}$ is stable, then the matrix sequence $\{X_ {k}\}$ generated by the Newton iteration converges quadratically to the unique positive semi-definite solution $X$ of \eqref{eq5}. In other words, there exists a constant $\delta>0$ independent of $k$, such that for all positive integers $$\|X_{k}-X\|_{2}\leq\delta\|X_{k-1}-X\|_{2}^{2}.$$ Furthermore, the iteration sequence $\{X_{k}\}$ exhibits monotonic convergence, i.e., $$0\leq X\leq\cdots\leq X_{k+1}\leq X_{k}\leq\cdots\leq X_{1}.$$
\end{theorem}
Let $$R(X_{k+1})=A_{k}^{T}X_{k+1}+X_{k+1}A_{k}-M_{k}M_{k}^{T}$$ be the residual matrix of the Lyapunov equation \eqref{eq17}, then we have the following proposition.
\begin{proposition}\label{proposition1}
Let $X_{k}$ be the $k$-th step iteration generated by the Kleinman-Newton method as described above, then$$F(X_{k+1})=-R(X_{k+1})-(X_{k}B-X_{k+1}B)(X_{k}B-X_{k+1}B)^{T}.$$
\end{proposition}
\begin{proof}
From the residual equation defined above, we can obtain that
\begin{align*}
R(X_{k+1})&=A_{k}^{T}X_{k+1}+X_{k+1}A_{k}-M_{k}M_{k}^{T}\\
&=(K_{k}B^{T}-A^{T})X_{k+1}+X_{k+1}(BK_{k}^{T}-A)-K_{k}K_{k}^{T}-C^{T}C\\
&=K_{k}B^{T}X_{k+1}-A^{T}X_{k+1}+X_{k+1}BK_{k}^{T}-X_{k+1}A-K_{k}K_{k}^{T}-C^{T}C\\
&=-F(X_{k+1})+K_{k}B^{T}X_{k+1}+X_{k+1}BK_{k}^{T}-X_{k+1}BB^{T}X_{k+1}-K_{k}K_{k}^{T}-C^{T}C\\
&=-F(X_{k+1})+X_{k}BB^{T}X_{k+1}+X_{k+1}BB^{T}X_{k}-X_{k+1}BB^{T}X_{k+1}-X_{k}BB^{T}X_{k}-C^{T}C\\
&=-F(X_{k+1})-(X_{k}-X_{k+1})BB^{T}(X_{k}-X_{k+1})\\
&=-F(X_{k+1})-(X_{k}B-X_{k+1}B)(X_{k}B-X_{k+1}B)^{T}.
\end{align*}
Therefore, from Proposition \ref{proposition1} we can directly prove
\begin{align*}
\|F(X_{k+1})\|_{2}&\leq\|R(X_{k+1})\|_{2}+\|(X_{k}B-X_{k+1}B)(X_{k}B-X_{k+1}B)^{T}\|_{2}\\
&\leq\|R(X_{k+1})\|_{2}+\|X_{k}B-X_{k+1}B\|_{2}^{2}.
\end{align*}
\end{proof}
\begin{theorem}\label{theorem6}
Let $X_{k+1}$ is a symmetric positive semi-definite  of the equation \eqref{eq17}, and let $(A_{k},\ M_{k})$ is stable, and parameter $\alpha>0,\ 0\leq\omega<2$, then for all $k=0,1,\cdots$, the low-rank form defined in equation \eqref{eq19} $V_{k+1}^{(l)}(W_{k+1}^{(l)})^{T}$ converges to $V_{k+1}W_{k+1}^{T}$, which is equivalent to the iteration sequence defined in equation \eqref{eq18} $\{X_{k+1}^{(l)}\}_{l=0}^{\infty}$ converges to $X_{k+1}$.
\end{theorem}
\begin{proof}
Here, we only need to prove the convergence of the GADI iteration format in equation \eqref{eq18}. Applying the flattening operator to equation \eqref{eq18} yields
\begin{equation}\label{eq20}
\left\{
\begin{array}{ll}
(\alpha_{k}I_{n^{2}}+I\otimes A_{k}^{T})\text{vec}(X_{k+1}^{(l+\frac{1}{2})})=(\alpha_{k}I_{n^{2}}-A_{k}^{T}\otimes I)\text{vec}(X_{k+1}^{(l)})+\text{vec}(Q_{k}),\\
(\alpha_{k}I_{n^{2}}+A_{k}^{T}\otimes I)\text{vec}(X_{k+1}^{(l+1)})=(A_{k+1}^{T}\otimes I-(1-\omega_{k})\alpha_{k}I_{n^{2}})\text{vec}(X_{k+1}^{(l)})+(2-\omega_{k})\alpha_{k}\text{vec}(X_{k+1}^{(l+\frac{1}{2})}),
\end{array}
\right.
\end{equation}
where $x=\text{vec}(X),\ \ q_{k}=\text{vec}(Q_{k})$, then the equivalent iteration format is obtained as follows
\begin{equation}\label{eq21}
\left\{
\begin{array}{ll}
(\alpha_{k}I_{n^{2}}+I\otimes A_{k}^{T})x_{k+1}^{(l+\frac{1}{2})}=(\alpha_{k}I_{n^{2}}-A_{k}^{T}\otimes I)x_{k+1}^{(l)}+q_{k},\\
(\alpha_{k}I_{n^{2}}+A_{k}^{T}\otimes I)x_{k+1}^{(l+1)}=(A_{k}^{T}\otimes I-(1-\omega_{k})\alpha_{k}I_{n^{2}})x_{k+1}^{(l)}+(2-\omega_{k})\alpha_{k}x_{k+1}^{(l+\frac{1}{2})}.
\end{array}
\right.
\end{equation}

Next, the proof of this theorem can be referenced to the proof method of Theorem 2.6 in \cite{ref23}.
\end{proof}
\subsection{Complexity analysis}
Here, we adopt the Kleinman-Newton iteration to transform the Riccati equation \label{eq05} into the Lyapunov equation \eqref{eq17} and then use the low-rank GADI method for iterative solving. Each iteration in the algorithm requires solving a Lyapunov equation, resulting in significant computational complexity. The Kleinman-Newton iteration method is an improved version of the Newton iteration method, as the right-hand side of equation \eqref{eq14}  is usually indefinite with a full rank matrix, while the method employed in this paper relies heavily on the low-rank structure of the right-hand side. The rank of the matrix selected in equation \label{eq16} is at most $m+p$, and we use the low-rank method to compute the approximate solution, making the Kleinman-Newton iteration method more suitable. We use the product of two low-rank matrices $V_{k}^{(l)}(W_{k}^{(l)})^{T}$ to replace $X_{k}^{(l)}$. Next, we provide the stopping criterion used in the algorithm.

Firstly, we define the residual matrix of the Riccati equation \label{eq05} as follows: $$F(X_{k})=A^{T}X_{k}+X_{k}A-X_{k}GX_{k}+Q.$$During the iteration process, in order to reflect its relative approximation level, the relative residual$$\text{Res}(X_{k})=\frac{\|A^{T}X_{k}+X_{k}A-X_{k}GX_{k}+Q\|_{2}}{\|Q\|_{2}}$$ is commonly used as a stopping criterion for iteration. However, in Newton iterations, forming the residual matrix  $R(V_{k}W_{k}^{T})$ requires a significant amount of memory. If the number of columns in $V_{k}$ is much smaller than the number of rows, the relative residual can be effectively used as the stopping criterion.

Additionally, we can observe the variation of the feedback matrix $K_{k+1}$. We use the criterion$$R_{s}=\frac{\|K_{k+1}-K_{k}\|_{2}}{\|K_{k+1}\|_{2}}<\varepsilon,$$
where $K_{k+1}=V_{k+1}W_{k+1}^{T}B$ and $\varepsilon$ is a very small positive number. This criterion is computationally efficient since $K\in\mathbb R^{n\times m}$ and $m\ll n$.

The direct iteration method has a memory requirement of $O(n^{2})$, and a computational complexity of $O(n^{3})$.  First, the computational cost of $X_{k+1}^{l+\frac{1}{2}}$ is $4n^{2}(2n-1)+3n^{2}$, and the cost of $X_{k+1}^{l+1}$ is $4n^{2}(2n-1)+4n^{2}$. Therefore, the total computational cost is $16n^{3}-n^{2}$. By comparing Algorithm \ref{algorithm2} and Algorithm \ref{algorithm3}, we can see that they have smaller memory requirements. If $A$ is a sparse matrix, the memory requirement can reach $O(n)$. Next, we calculate their computational complexity. Since $A\in\mathbb R^{n\times n}, \ B\in\mathbb R^{n\times m}$ and $C\in\mathbb R^{p\times n}$, during the iteration process, the column numbers of matrices $V_{k}$ and $W_{k}$ ncrease with the number of iterations. Here, we have $A_{k}\in\mathbb R^{n\times n},\ M_{k}\in\mathbb R^{n\times(m+p)}$. The computational cost of calculating $V_{1}^{(1)}$ is $2n^{2}p$, where $W_{1}^{(1)}=V_{1}^{(1)}$. The cost of calculating $V_{k}^{(1)}$ is $2n^{2}(m+p)$, and the cost of $V_{k}^{(l)}$ is $n^{2}(2^{l}-2)(m+p)$, while the cost of $W_{k}^{(l)}$ is $4n^{3}+2n^{2}(2^{l}-2)(m+p)$. Therefore, the total computational cost is $4n^{3}+n^{2}(3\cdot2^{l}(m+p)-2(2m+p))$.

\subsection{Numerical experiments}\

\textbf{Example 3.4.1\ \ }\ We take the coefficient matrix of the Riccati equation \eqref{eq5} as$$A=
\left(
\begin{array}{ccccc}
    -12   &  -3   &  0   &  \cdots   &  0\\
     2 &  -12   &  -3  &   \cdots   &  0\\
     \vdots   &  \ddots  &  \ddots  &   \ddots   &  \vdots\\
     0   &  \cdots   &  2  &   -12   &  -3\\
     0   & \cdots   &  0  &   2  &   -12\\
  \end{array}
\right)_{n\times n},\ \ B=\left(
                         \begin{array}{c}
                           0.2 \\
                           0.2 \\
                           \vdots \\
                           0.2 \\
                           0.2 \\
                         \end{array}
                       \right)_{n\times1},$$
                       $$ C=\left(
                     \begin{array}{ccccc}
                       0.1, & 0.1 ,& \cdots ,& 0.1 ,& 0.1 \\
                     \end{array}
                   \right)_{1\times n},$$
with $G=BB^{T},\ \ Q=C^{T}C$.

When the matrix dimension increases in multiples, we choose the relative residual as the iteration stopping criterion and use the Kleinman-Newton-GADI (K-N-GADI), Kleinman-Newton-RADI (K-N-RADI), and Kleinman-Newton-RGADI (K-N-RGADI) methods for computation. These iteration methods all start from the initial value $X_{0}=0$ and yield numerical results shown in Table \ref{table4}. From the table data, we can see that the K-N-RGADI method is more efficient in solving this example problem compared to the K-N-GADI and K-N-RADI methods. Additionally, Fig. \ref{figure5} clearly shows the variation of iteration steps and relative residual for these three methods when $n=1024$, while Fig. \ref{figure5} displays the time consumption of these three iteration methods as the matrix dimension increases. These findings further demonstrate the effectiveness of the K-N-RGADI method.
\begin{table}[!htbp]
 \centering
\begin{tabular}{ccccc}
   \hline
   $n$ &algorithm  & Res &out(int)IT &CPU\\
  \hline
   128& K-N-GADI & 4.2542e-15 & 8(8) & 0.36s  \\
    128& K-N-RADI  & 1.7696e-15 & 4(8) & 0.24s  \\
     128&K-N-RGADI  & 2.6821e-15 & 4(8)& 0.21s \\
    \hline
   256& K-N-GADI  & 6.9046e-15 & 8(8) & 1.38s  \\
    256& K-N-RADI & 8.5988e-15 & 4(8) & 0.76s  \\
     256&K-N-RGADI  & 5.0362e-15 & 4(8)& 0.69s \\
     \hline
   512& K-N-GADI  & 7.1477e-15 & 8(9) & 11.73s  \\
    512& K-N-RADI  & 1.0959e-14 & 6(9) & 7.13s  \\
     512&K-N-RGADI  & 8.9506e-15 & 6(9)& 6.85s \\
     \hline
   1024& K-N-GADI  & 9.2023e-15 & 8(13) & 172.32s  \\
    1024& K-N-RADI & 4.2303e-14 & 6(11) & 141.4s  \\
     1024&K-N-RGADI & 5.914e-15 & 6(12)& 134.96s \\
     \hline
   2048& K-N-GADI  & 7.4253e-12 &10(16) & 3412.4s  \\
    2048& K-N-RADI & 6.038e-13 & 8(16) & 2941.4s  \\
     2048&K-N-RGADI  & 2.1016e-13 & 8(16)& 2805s \\
   \hline
 \end{tabular}
 \caption{Numerical results for Example 3.4.1\label{table4}}
 \end{table}
  \begin{figure}[H]
\centering
    \begin{minipage}[t]{0.49\textwidth}
        \centering
        \includegraphics[width=1.1\textwidth]{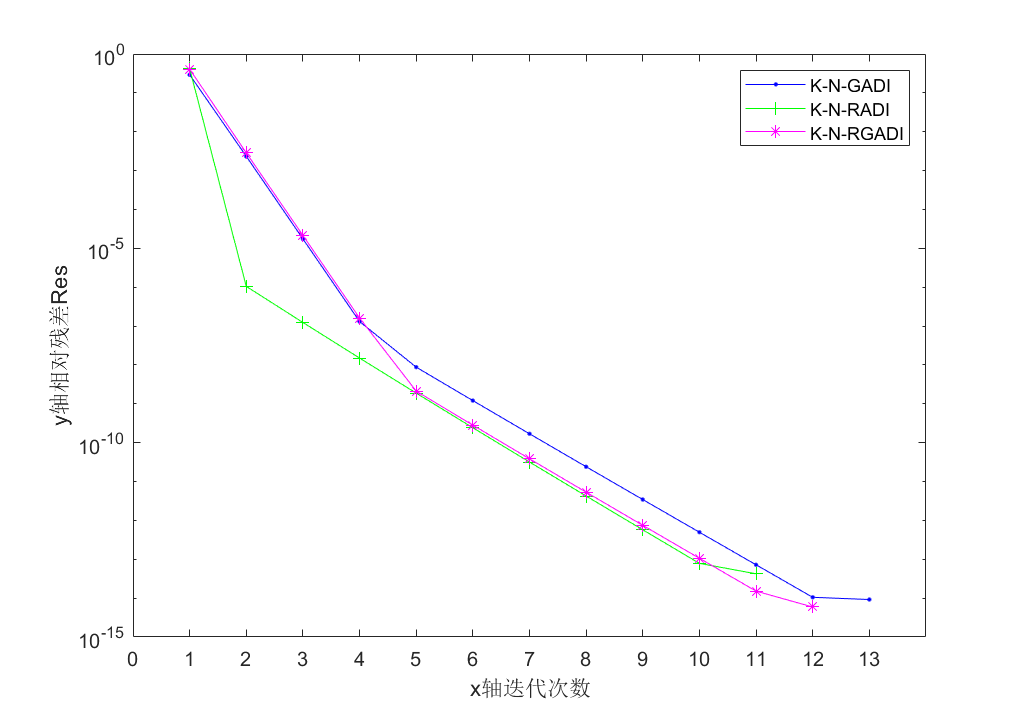}
        \caption{The residual curve of Example 3.4.1~$n=1024$\label{figure5}}
        \label{}
    \end{minipage}
    \begin{minipage}[t]{0.49\textwidth}
        \centering
        \includegraphics[width=1.25\textwidth]{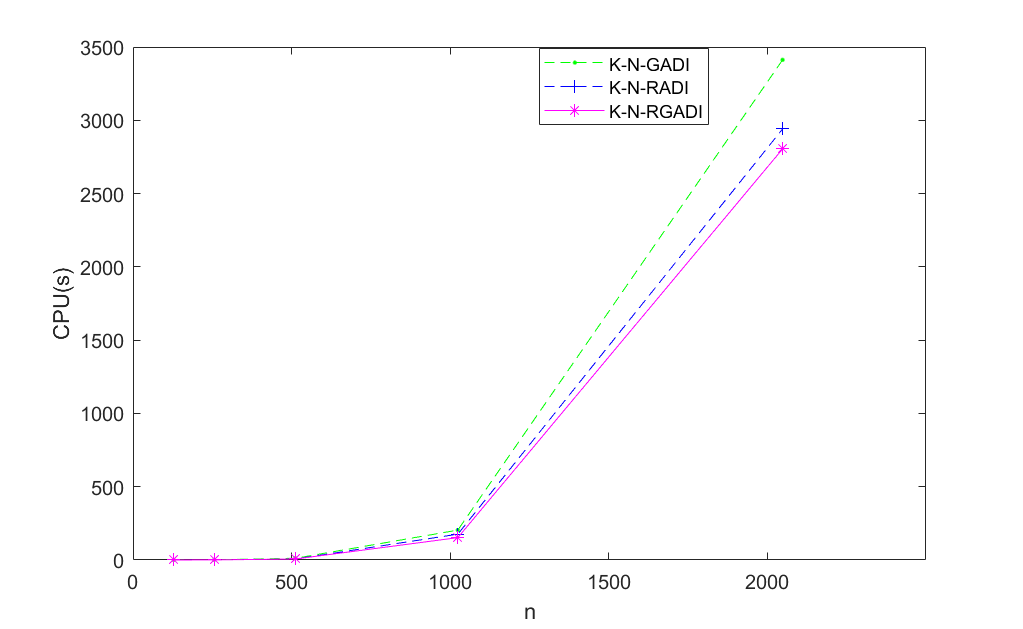}
        \caption{The time curve of Example 3.4.1\label{figure6}}
        \end{minipage}
\end{figure}

Furthermore, as the order $n$ of the coefficient matrix in the equation increases multiplicatively, we employ the K-N-RGADI method with the relative change of the feedback matrix $(R_{s})$ as the iteration stopping criterion. We compare this method with the K-N-GADI and K-N-RGADI methods that use the relative residual Res as the iteration stopping criterion. From Table \ref{table5}, we can observe that when using the relative change of the feedback matrix as the iteration stopping criterion, the K-N-RGADI method significantly reduces the running time.

\begin{table}[!htbp]
\begin{tabular}{|c|c|c|c|c|c|c|}
   \hline
   \multicolumn{7}{|c|}{K-N-GADI uses Res as the iteration stop standard}\\
    \hline
    $n$ & 128 & 256 & 512 &1024& 2048 &4096\\
   \hline
    CPU & 0.36s &1.38s &11.73s &172.42s& 3412.4s &-- \\
    \hline
    \multicolumn{7}{|c|}{K-N-RGADI uses Res as the iteration stop standard}\\
    \hline
    $n$ & 128 & 256 & 512 &1024& 2048 &4096\\
  \hline
    CPU & 0.21s &0.69s &6.85s &134.96s& 2805s &-- \\
     \hline
    \multicolumn{7}{|c|}{K-N-RGADI uses $R_{s}$ as the iteration stop standard}\\
    \hline
    $n$ & 128 & 256 & 512 &1024& 2048 &4096\\
   \hline
    CPU & 0.12s & 0.57s & 2.26s & 30.95s & 213.34s & 3478.8s\\
   \hline
 \end{tabular}
 \centering
 \caption{Numerical results for Example 3.4.1\label{table5}}
 \end{table}\

Fig. \ref{figure7} shows the iterative steps and corresponding residual $R_{s}$ obtained by using the K-N-RGADI method to compute different matrix dimensions $n$. It can be clearly seen that as the matrix dimension increases, the residual $R_{s}$ also increases.
\begin{figure}[H]
  \center
  \includegraphics[width=12cm,height=8cm] {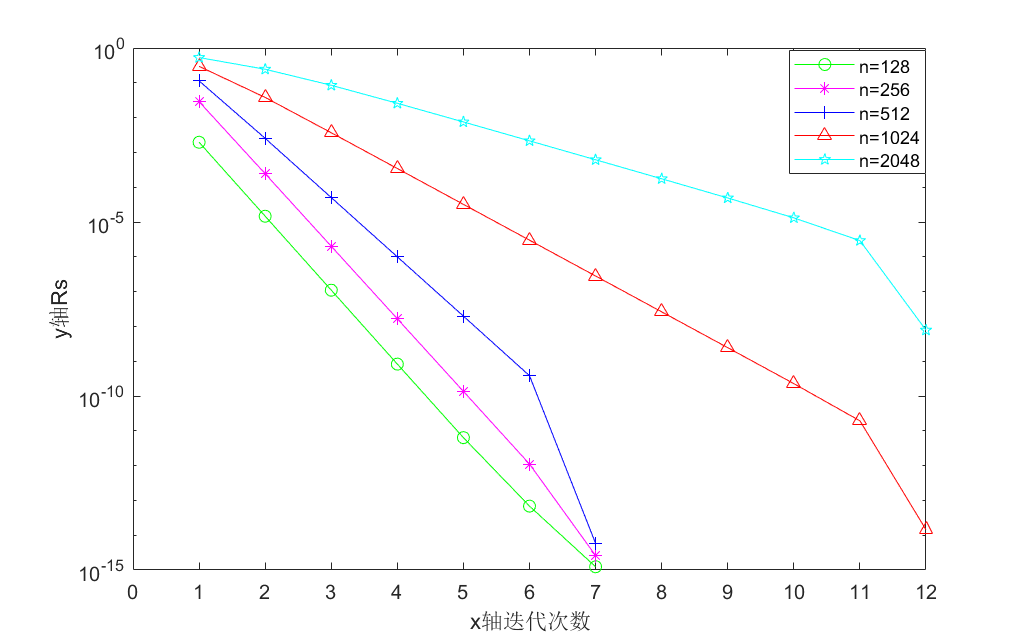}\\
  \caption{The residual curve of Example 3.4.1\label{figure7}}
\end{figure}

\textbf{Example 3.4.2\ \ }\ We take the coefficient matrix of the Riccati equation \eqref{eq5} as$$A=
\left(
\begin{array}{ccccccc}
    -12   &  -3   &  -2 & 0  & 0 &  \cdots   &  0\\
     2 &  -12   &  -3  &  -2 &  0 &\cdots   &  0\\
     1 &  2   &  -12  &  -3  & -2&  \cdots   &  0\\
     \vdots   &  \ddots  &  \ddots & \ddots & \ddots &   \ddots   &  \vdots\\
     0  &  \cdots  & 1  &  2  &   -12  & -3 &  -2\\
     0 &  \cdots &0 & 1  &  2  &   -12   &  -3\\
     0   & \cdots  &0 &  0   & 1 &  2  &   -12\\
  \end{array}
\right)_{n\times n},\ \ B=\left(
                         \begin{array}{c}
                           0.2 \\
                           0.2 \\
                           \vdots \\
                           0.2 \\
                           0.2 \\
                         \end{array}
                       \right)_{n\times1},$$
                       $$C=\left(
                     \begin{array}{ccccc}
                       0.1, & 0.1 ,& \cdots ,& 0.1 ,& 0.1 \\
                     \end{array}
                   \right)_{1\times n}.$$

We first used the relative residual Res as the iterative stopping criterion and adopted K-N-GADI, K-N-RADI, and K-N-RGADI methods to solve this example problem. The initial iteration value is set to $X_{0}=0$, and the numerical results obtained are shown in Table \ref{table6}.

\begin{table}[!htbp]
\begin{tabular}{ccccc}
   \hline
   $n$ &algorithm  & Res &out(int)IT &CPU\\
   \hline
   128& K-N-GADI & 2.2313e-15 & 8(10) & 0.46s  \\
    128& K-N-RADI  & 2.2274e-15 & 4(10) & 0.35s  \\
     128&K-N-RGADI  & 3.4297e-15 & 4(10)& 0.32s \\
     \hline
   256& K-N-GADI  & 4.8395e-15 & 8(10) & 1.69s  \\
    256& K-N-RADI & 9.4654e-15 & 4(10) & 1.06s  \\
    256 &K-N-RGADI & 6.6721e-15 & 4(10)& 1.04s \\
    \hline
   512& K-N-GADI  & 1.3522e-14 & 8(9) & 13.68s  \\
    512& K-N-RADI & 1.3458e-14 & 6(10) & 10.56s  \\
    512 &K-N-RGADI & 1.22e-14 & 6(10)& 8.29s \\
     \hline
   1024& K-N-GADI  & 7.2204e-14 & 8(13) & 203.93s  \\
    1024& K-N-RADI & 2.9667e-14 & 6(12) & 176.04s  \\
     1024&K-N-RGADI& 2.0719e-14 & 6(12)& 152.61s \\
     \hline
   2048& K-N-GADI  & 1.2224e-12 &10(16) & 4019.5s  \\
    2048& K-N-RADI  & 2.5904e-13 & 8(16) &  3145.8s  \\
    2048 &K-N-RGADI  & 3.2006e-13 & 8(16)& 2988.2s \\
   \hline
 \end{tabular}
 \centering
 \caption{Numerical results for Example 3.4.2\label{table6}}
 \end{table}

Next, we utilize the relative change of the feedback matrix $(R_{s})$ as the stopping criterion for the K-N-RGADI iteration method and compare its runtime with that of the K-N-GADI and K-N-RGADI iteration methods using the relative residual Res as the stopping criterion. As shown in Table \ref{table7}, when using the relative change of the feedback matrix as the iteration stopping criterion, the K-N-RGADI method achieves a shorter runtime.

\begin{table}[!htbp]
\begin{tabular}{|c|c|c|c|c|c|c|}
   \hline
   \multicolumn{7}{|c|}{K-N-GADI uses Res as the iteration stop standard}\\
    \hline
    $n$ & 128 & 256 & 512 &1024& 2048 &4096\\
   \hline
    CPU & 0.46s &1.69s &13.68s &203.93s& 4019.5s &-- \\
    \hline
    \multicolumn{7}{|c|}{K-N-RGADI uses Res as the iteration stop standard}\\
    \hline
    $n$ & 128 & 256 & 512 &1024& 2048 &4096\\
   \hline
    CPU & 0.32s &1.04s &8.29s &152.61s& 2988.2s &-- \\
     \hline
    \multicolumn{7}{|c|}{K-N-RGADI uses $R_{s}$ as the iteration stop standard}\\
   \hline
    $n$ & 128 & 256 & 512 &1024& 2048 &4096\\
   \hline
    CPU & 0.14s & 0.63s & 2.58s & 35.26s & 218.98s &3519.6.5s\\
   \hline
 \end{tabular}
 \caption{Numerical results for Example 3.4.2\label{table7}}
 \centering
 \end{table}

\section{Conclusions}
\label{sec:conclusion}
This paper presents a low-rank GADI algorithm for computing low-rank approximate solutions to large-scale Lyapunov and algebraic Riccati equations. In the computation of low-rank approximate solutions to the algebraic Riccati equation, we combine the Kleinman-Newton method and utilize the low-rank GADI algorithm to solve the Lyapunov equation at each Newton step, resulting in the Kleinman-Newton-RGADI algorithm. Additionally, we observe that the low-rank GADI method exhibits the same convergence properties as the GADI method when solving both the Lyapunov and algebraic Riccati equations with low-rank approximations. Furthermore, numerical examples are provided to compare the effectiveness of the low-rank ADI algorithm and the low-rank GADI algorithm. The results demonstrate that the low-rank GADI method is more efficient. However, like other solvers, the performance of this algorithm heavily relies on the choice of shift parameters, which remains a challenging problem.





\end{document}